\documentclass[11pt]{article}

\usepackage{url,subfig}
\usepackage[left=1in,top=1in,right=1in,bottom=1in,nohead]{geometry}

\usepackage{latexsym,amsmath,amssymb, amsfonts,enumerate,epsfig, graphics,times,color, amsthm,multicol}

\newtheorem{theorem}{Theorem}
\theoremstyle{definition}
\newtheorem{lemma}[theorem]{Lemma}

\newtheorem{algorithm}{Algorithm}
\newtheorem{example}{Example}

\newtheorem{num_ex}{Numerical Example}

   \newtheorem{assumption}{Assumption}

\newcommand{\eqdef}{\overset{\mbox{\tiny def}}{=}}
\newcommand{\R}{\mathbb{R}}
\newcommand{\Z}{\mathbb{Z}}
\newcommand{\E}{\mathbb{E}}
\newcommand{\eqdist}{\overset{\mathcal{D}}{=}}
\newcommand{\EG}{\mathcal{A}}

\title{An Efficient Finite Difference Method for Parameter Sensitivities of Continuous Time Markov Chains}

\author{David F. Anderson\footnote{Department of Mathematics, University of
  Wisconsin, Madison, Wi. 53706, anderson@math.wisc.edu, grant support from NSF-DMS-1009275.}}

\begin{document}

\maketitle

\begin{abstract}

We present an efficient finite difference method for the computation of parameter sensitivities that is applicable to a wide class of continuous time Markov chain models.   The estimator for the method is constructed by coupling the perturbed and nominal processes in a natural manner, and the analysis proceeds by utilizing a martingale representation for the coupled processes. The variance of the resulting estimator is shown to be an order of magnitude lower  due to the coupling.   We conclude that the proposed method produces an estimator with a lower variance than other methods, including the use of Common Random Numbers, in most situations.  Often the variance reduction is substantial. The method is no harder to implement than any standard continuous time Markov chain algorithm, such as ``Gillespie's algorithm.''  The motivating class of models, and the source of our examples,  are the stochastic chemical kinetic models commonly used in the biosciences, though other natural application areas include  population processes and queuing networks.

\end{abstract}

  \footnotetext[0]{AMS 2000 subject classifications: Primary 60H35, 65C99; Secondary 92C40}

\noindent
\textbf{Keywords:} Finite differences, variance reduction, parameter sensitivities,  next reaction method, Gillespie, random time change, continuous time Markov chain, martingale. 

\section{Introduction}
\label{sec:intro}

We develop a new finite difference method for the computation of parameter sensitivities that is applicable to a wide class of continuous time Markov chain models.  For $k \in \{1,\dots,M\}$,  let $\zeta_k\in \R^d$ denote the possible transition directions for a continuous time Markov chain, and let $\lambda_k:\R^d \to \R$ denote the respective intensity, or propensity functions.\footnote{In the language of probability theory, the functions $\lambda_k$ are nearly universally termed intensity functions, whereas in the biosciences they are nearly universally termed propensity functions.}
  The  random time change representation of Kurtz for the model is 
\begin{equation}
	X(t) = X(0) + \sum_{k = 1}^M Y_k\left( \int_0^t \lambda_k(X(s)) ds\right) \zeta_k,
	\label{eq:main}
\end{equation}
where the $Y_k$ are independent, unit-rate Poisson processes.  See, for example, \cite{KurtzPop81},
\cite[Chapter 6~]{Kurtz86}, or the recent survey \cite{AndKurtz2011}.  The infinitesimal generator for the model \eqref{eq:main} is the operator $\EG$ satisfying
\begin{equation*}
  (\EG f)(x) = \sum_{k} \lambda_k(x)(f(x + \zeta_k) - f(x)),
\end{equation*}
where $f : \R^d \to \R$ is chosen from a sufficiently large class of functions.  Without loss of generality, we assume throughout that the state space of the process, $\mathcal S$, is a subset of $\Z^d$.

Consider a family of models \eqref{eq:main} indexed by a set of parameters, which we denote by the vector $\theta$.  Even when there are good theoretical reasons for believing the model is a reasonable description of some phenomenon, usually the parameters are not known precisely and have to be estimated experimentally. Depending on the setup and the parameters in question, it may be  difficult to obtain good estimates. Thus, it is important to analyze how sensitive features of interest in the model are to variation in the  parameters. 
For ease of exposition we take $\theta$ to be a scalar, though note that it is trivial to extend all of the ideas of the paper to the setting of $\theta \in \R^\ell$, for some $\ell > 0$.  

We let $f : \mathcal{S} \to \R$ be a function of the state of the system that gives a measurement of interest.   For example, $f$ could be the abundance of one of the components at a particular time.  Define $J(\theta) \eqdef \E f(X^{\theta}(t)),$ where the $\theta$ dependence is being made explicit.  The problem of interest is to efficiently approximate $J'(\theta)$.

 There are a number of methods that can be used for the computation of such parameter sensitivities in this setting, including finite differences, likelihood ratios and Girsanov transformations, and infinitesimal perturbation analysis, each with its own benefits and drawbacks; see for example  \cite{GlynnAsmussen2007,Komorowski2011, Plyasunov2007, Khammash2010}.  We  focus on finite difference methods, which due to its simplicity, is the most popular choice.   Specifically, in this paper a new finite difference method is introduced that is easy to implement, analytically tractable, and typically produces an estimate with a given tolerance with substantially lower computational complexity than that obtained using the other methods currently known to the author.

While  continuous time Markov chain models of the general form  \eqref{eq:main} are used ubiquitously in both industry and the sciences to model natural phenomenon ranging from population processes to queueing networks, we feel the method developed here will be most useful in the study of stochastic models of biochemical reaction networks.  We will therefore choose the language of biochemistry throughout, and also choose this area as the setting for our examples.

 A biochemical reaction network is a chemical system involving multiple reactions and chemical species.   If the abundances of the constituent molecules of a reaction network
are sufficiently high then their concentrations are typically modeled by
a coupled set of ordinary differential equations.  If, however, the
abundances are low then the standard deterministic models do not
provide a good representation of the behavior of the system and
stochastic models are used.   The simplest stochastic models of such networks \cite{Kurtz72, McQuarrie67}  treat the system as a continuous time Markov chain with the state, $X$, being the number of molecules of each species and with reactions modeled as possible transitions of the chain.  More explicitly, if the $k$th reaction happens at time $t$, then the system is updated by the {\em reaction vector} $\zeta_k$,
 \begin{equation*}
 	X(t) = X(t-) + \zeta_k.
 \end{equation*}
 Letting $\lambda_k: \R^d \to \R$ denote the intensity, or propensity, of the $k$th reaction, we see this stochastic model satisfies \eqref{eq:main}.

 As will be pointed out in the following sections, the strategy being proposed here is in some ways similar to the common reaction path (CRP) method  proposed  in \cite{Khammash2010}, which is also a quite capable estimator for finite differences.  However, there are important differences.  First, the actual coupling, and hence simulation, of the relevant  processes is different.  The coupling proposed here tends to provide an estimator with a lower variance, especially when the process is considered for moderate to large times.
 Second,  the coupling proposed here lends itself to analysis more readily than that used in \cite{Khammash2010} as the centered counting processes used in our coupling are martingales with respect to the natural filtration of the process \cite{AndersonGangulyKurtz}. Third, the method being proposed here is as easy to implement as the usual Gillespie algorithm or next reaction method.  The strategy employed in \cite{Khammash2010}, on the other hand, requires being quite careful with the seeds of the pseudo-random number generators used since one independent seed is required per reaction channel per sample path generated.   Finally, the coupling proposed here essentially converts the problem of generating two paths of a continuous time Markov chain into a problem of generating one path of a different continuous time Markov chain with an enlarged state space.  Therefore, all analytical and computational techniques developed for the study of continuous time Markov chains, of which there are many, will  be employable on this larger system, and therefore applicable to the problem of computing sensitivities.  
  
 The most common finite difference coupling used today for the approximation of sensitivities is probably an implementation of Gillespie's algorithm plus using common random numbers (CRN).  We will also discuss this coupling and give the relevant stochastic representation of it.  We will  conclude that the coupling proposed here will produce a lower variance estimator than CRN for the same reasons that it produces a lower variance estimator than the CRP scheme of \cite{Khammash2010}.
 
  The main goals of this paper are to introduce the new method and to provide the mathematical analysis of the expected squared difference between the two relevant processes, though some relevant examples will also be provided.  
 In Section \ref{sec:formal_model}, we formally introduce our mathematical model of interest, including all technical assumptions.     In Section \ref{sec:new_methods}, we develop our new finite difference estimator and provide sharp analytical bounds. We also discuss the long time behavior of the introduced estimator and compare it with both CRP and CRN.  We conclude that the proposed method will be quite superior for moderate and large time scales.   In Section \ref{sec:examples},  we provide  examples demonstrating our main results. 
   

\section{The Formal Setup}
\label{sec:formal_model}

We consider the family of models
\begin{equation}
	X^{\theta}(t) = X^{\theta}(0) + \sum_{k = 1}^M Y_k\left( \int_0^t \lambda^{\theta}_k(X^{\theta}(s)) ds\right) \zeta_k,
	\label{eq:main_indexed}
\end{equation}
where the $Y_k$ are independent, unit-rate Poisson processes, the vector $\theta$ represents a given choice of parameters that we are making explicit in the notation, and all other notation is as before.  The assumption that there are a finite number of possible jump directions $\zeta_k$ can almost certainly be weakened.  However, this assumption makes the analysis significantly cleaner and all the motivating models (such as those arising from biochemistry) naturally satisfy such a condition.  We define $F^{\theta}:\Z^d \to \R^d$ by
\begin{equation*}
	F^{\theta}(z) \eqdef \sum_k \lambda_k^{\theta}(z)\zeta_k.
\end{equation*}	
 We make the following running assumptions throughout the remainder of the paper.  The first is that the intensity functions are uniformly (in $\theta$) globally Lipschitz.  The second is that the intensity functions scale at least linearly with perturbations to $\theta$.

\begin{assumption}	\label{assump:GL}
	We suppose that there is a $K_1>0$ for which 
	\begin{equation*}
		|\lambda_k^{\theta}(x) - \lambda_k^{\theta}(y)| +  |F^{\theta}(x) - F^{\theta}(y)|  \le K_1|x-y|, 
	\end{equation*}
	for all $k$, $\theta$ of interest, and all $x,y\in \mathcal S$. 
\end{assumption}

\begin{assumption}\label{assump:theta}
	We suppose there is a $K_2>0$ so that for all $k$ and all $\epsilon<1$, 
	\begin{equation*}
		\sup_{x\in{\mathcal S}} \left[ | \lambda_k^{\theta + \epsilon}(x) - \lambda_k^{\theta}(x)|+ |F^{\theta+\epsilon}(x) - F^{\theta}(x)|\right] \le K_2 \epsilon.
	\end{equation*}
\end{assumption}

Assumption \ref{assump:GL} can almost certainly be weakened to a local Lipschitz condition, in which case analytical methods similar to those found in \cite{Hutzenthaler2011} and/or \cite{MattinglyStuart} can be applied.  Proving our main results in such generality, while possible and certainly worth doing in future work, will be significantly messier and we feel the main points of the analysis will be lost.   
Note that  Assumption \ref{assump:GL} automatically holds if $\mathcal S$ is a bounded set.  In the chemical setting, mild assumptions on the intensity functions ensure that the non-negative orthant is forward invariant.  Therefore, $\mathcal S$ is bounded if there is a vector $\xi \in \Z^d_{>0}$ for which $\xi\cdot \zeta_k \le 0$ for all $k$.  For example, such a $\xi$ exists if mass is conserved.   Assumption \ref{assump:GL}  also holds if the intensity functions are simply set to zero outside of a compact subset of $\Z^d$, which has the same effect as analyzing the standard model up until a stopping time $\tau$, defined to be the time the processes leaves a given compact set (essentially using a localization argument).

Another relevant situation in which Assumption \ref{assump:GL} holds is when the equivalent \textit{scaled} models are analyzed.  While we point the reader to \cite{AndersonGangulyKurtz, AndersonHigham2011, AndMaso2011, KangKurtz2011} for a thorough description of this model in the biosciences, we will briefly discuss it here.  For some parameter of the system, $N$, we let $X^N$ denote the process with $i$th component $X_i^N = X_i/N^{\alpha_i}$, where the $\alpha_i \ge 0$ are chosen so that $X_i^N = O(1)$.  Under mild assumptions on the intensity functions $\lambda_k$, it can be shown that $X^N$ satisfies 
\begin{equation}\label{eq:main_multi}
	X^N(t) = X^N(0) + \sum_k Y_k\left( \int_0^t N^{\beta_k + \nu_k \cdot \alpha} \lambda_k(X^N(s)) ds\right) \zeta_k^N,
\end{equation}
where $\zeta_{k,i}^N = \zeta_{k,i}/N^{\alpha_i}$, and $\beta_k$ is chosen so that $ \lambda_k(X^N(\cdot)) = O(1)$.  This scaled model is $O(1)$ and therefore more readily satisfies Assumption \ref{assump:GL}, where now it is understood that the state space is
\begin{equation*}
	{\mathcal S}^N \eqdef \{z \in \R^d\ | \ z_i = x_iN^{-\alpha_i}, x \in {\mathcal S}\}.
\end{equation*}
In many applications, it is more relevant to compute expectations and parameter sensitivities of the scaled model \eqref{eq:main_multi} than the unscaled version \eqref{eq:main}, though we do not revisit this point in the current paper.

\subsection{The basic problem and the benefits of variance reduction}
\label{sec:theWhy}

We let $f : \Z^d_{\ge 0} \to \R$ be a function of the state of the system which gives a measurement of interest and define 
 \begin{equation*}
   J(\theta) \eqdef \E f(X^{\theta}(t)).
\end{equation*}
 The problem of interest is to efficiently estimate $J'(\theta)$, where we recall that we are making the simplifying assumption  that $\theta$ is one-dimensional.
 
 To estimate $J'(\theta)$  the centered finite difference is often used:  
\begin{equation}
	J'(\theta) \approx \frac{\E f(X^{\theta + \epsilon/2}(t)) - \E f(X^{\theta - \epsilon/2}(t))}{\epsilon},
	\label{eq:centered}
\end{equation}
as its bias is $O(\epsilon^2)$ \cite{GlynnAsmussen2007}.  That is,
\begin{align*}
	J'(\theta) &= \frac{\E f(X^{\theta + \epsilon/2}(t)) - \E f(X^{\theta - \epsilon/2}(t))}{\epsilon} + O(\epsilon^2).
\end{align*}
This should be compared with the forward difference, which has a bias of $O(\epsilon)$
\begin{align*}
	J'(\theta) &= \frac{\E f(X^{\theta + \epsilon}(t)) - \E f(X^{\theta}(t))}{\epsilon} + O(\epsilon).
\end{align*}
The estimator for \eqref{eq:centered} using centered finite differences is 
\begin{equation}
D_R(\epsilon) = \frac{1}{R} \sum_{i = 1}^R d_{[i]}(\epsilon), 
\label{eq:difference}
\end{equation}
with
\begin{equation}
	d_{[i]}(\epsilon)= \frac{f(X_{[i]}^{\theta + \epsilon/2}(t)) - f(X_{[i]}^{\theta - \epsilon/2}(t))}{\epsilon},
	\label{eq:d}
\end{equation}
where $X_{[i]}^{\theta}$ represents the $i$th path generated with parameter choice $\theta$, and $R$ is the number of paths generated.  If $X_{[i]}^{\theta + \epsilon/2}(t)$ and $X_{[i]}^{\theta - \epsilon/2}(t)$ are computed independently, the variance of $d_{[i]}(\epsilon)$ is $O(\epsilon^{-2})$, and, hence, the variance of $D_R(\epsilon)$ is $O(R^{-1} \epsilon^{-2})$.  Note that
\begin{align*}
	\E (D_R(\epsilon) - J'(\theta))^2 &= \mathsf{Var}(D_R(\epsilon)) + (\E D_R(\epsilon) - J'(\theta))^2\\
	&= O\left(R^{-1}\epsilon^{-2}\right) + O\left(\epsilon^4\right),
\end{align*}
which, for a given $R$, is minimized when $\epsilon = O(R^{-1/6})$, at a value that is $O(R^{-2/3})$.  Therefore, the optimal convergence rate to the exact value, in the sense of confidence intervals, is $O(R^{-1/3})$  \cite{GlynnAsmussen2007}.  

Many computations are performed with a target variance (which yields a target size of the confidence interval).  Denoting the target variance by $V^*$, we see that the number of paths required is then approximated by the  solution to
\[
	\mathsf{Var}\left( \frac{1}{R} \sum_{i = 1}^R d_{[i]}(\epsilon) \right) = \frac{1}R \mathsf{Var}(d(\epsilon)) = V^* \implies R = \frac{1}{V^*}\mathsf{Var}(d(\epsilon)).
\]
Thus,   decreasing the variance of $d(\epsilon)$  lowers the computational complexity (total number of computations)  required to solve the problem.
The basic idea of coupling, in the context of this paper, is to lower the variance of $d(\epsilon)$ by simulating $X^{\theta + \epsilon/2}$ and $X^{\theta - \epsilon/2}$ simultaneously so that the two processes are highly correlated or ``coupled.''  That is, instead of generating paths independently, we want to generate a \textit{pair} of paths $(X^{\theta + \epsilon/2}, X^{\theta - \epsilon/2})$ so that for appropriate choices of $f$, the variance of $f(X^{\theta + \epsilon/2}) - f(X^{\theta - \epsilon/2})$ is reduced.  The basic idea of any such coupling is to reuse, or share, some portion of the driving ``noise'' in the generation of each process.  As already alluded to in the Introduction, one such finite difference method that achieved a substantial reduction in variance due to coupling can be found in \cite{Khammash2010}, which we discuss in more detail in later sections.

In Section \ref{sec:results}, we will develop a new coupling technique so that the variance of $d_{[i]}(\epsilon)$ in  \eqref{eq:d} is $O(\epsilon^{-1})$, a full order of magnitude lower (in $\epsilon$) than when the paths were generated independently.  This will lead to a finite difference method  with an optimal convergence rate, in the sense of the above paragraph, of $O(R^{-2/5})$, achieved when $\epsilon = O(R^{-1/5})$.  More importantly, however, the variance of the estimator \eqref{eq:difference} will be $O(R^{-1}\epsilon^{-1})$, which should be compared with a variance of $O(R^{-1}\epsilon^{-2})$ when independent paths are used.  Thus, the number of paths (and  computational complexity) required to solve a given problem will be reduced by an order of $\epsilon$.

\section{Coupled finite differences}
\label{sec:new_methods}

In  Section \ref{sec:results}, we discuss how to couple the requisite processes for the coupled finite difference method being proposed here.  In Section \ref{sec:analytical_results}, we provide sharp bounds on the variance of the estimator.

 \subsection{Coupling the processes}
\label{sec:results}

Whether using the forward or centered difference, the main problem is to intelligently produce two paths generated from systems whose parameters differ by an order of $\epsilon$.  A good coupling should satisfy three things:
\begin{enumerate}[$(i)$]
\item it should minimize the variance of the difference \eqref{eq:d},
\item it should be easy to simulate, and 
\item it should be analytically tractable.  
\end{enumerate}
We will show that the coupling \eqref{eq:main_coupling} below satisfies each of these requirements, however we begin by motivating the coupling by two simpler problems that capture the core idea.  

We consider the problem of trying to understand the difference between $Z_1(t)$ and $Z_2(t)$, where $Z_1,Z_2$ are Poisson processes with rates $13.1$ and $13$, respectively.  We let $Y_1$  and $Y_2$ be independent unit-rate Poisson processes, and set 
\begin{align*}
	Z_1(t) &= Y_1(13 t) + Y_2(0.1 t)\\
	Z_2(t) &= Y_1(13t),
\end{align*}
where we use the additivity property of Poisson processes.  The important point to note is that both processes $Z_1$ and $Z_2$ are using the process $Y_1(13t)$ to generate simultaneous jumps.  The process $Z_1$ then uses the auxiliary process $Y_2(0.1 t)$ to jump the extra times that $Z_2$ does not.  The processes $Z_1$, $Z_2$ will jump together the vast majority of times, and in this way be very tightly coupled.\footnote{In this case, the long-run percentage of jumps that are shared can be quantified precisely as $13/(13 + 0.1) \approx 0.99923$.}  The coupling above also already hints at the main points of the mathematical analysis that will be carried out in Section \ref{sec:analytical_results} as
\begin{align*}
	Z_1(t) - Z_2(t) = Y_2(0.1t),
\end{align*}
and so,
\begin{align*}
	\E | Z_1(t) - Z_2(t)| &= \E Y_2(0.1t) = 0.1t\\
	\E (Z_1(t) - Z_2(t))^2 &= \E Y_2(0.1t)^2 = 0.1t + 0.01t^2.
\end{align*}

More generally, if $Z_1$ and $Z_2$ are  
non-homogeneous Poisson processes with intensities 
$f(t)$ and $g(t)$, respectively, then we could let 
$Y_1$, $Y_2$, and $Y_3$ be independent, 
unit-rate Poisson processes and define
\begin{align*}
	Z_1(t) &= Y_1\left(\int_0^t f(s) \wedge g(s) ds\right) 
       + Y_2\left( \int_0^t f(s) - 
           \left( f(s) \wedge g(s)  \right) ds\right),\\
	Z_2(t) &=  Y_1\left(\int_0^t f(s) \wedge g(s) ds\right) 
     + Y_3\left( \int_0^t g(s) - \left(f(s) \wedge g(s)  
                           \right)ds\right),
\end{align*}
where we are using that, for example,
\begin{equation*}
Y_1\left(\int_0^t f(s) \wedge g(s) ds\right) + 
        Y_2\left( \int_0^t f(s) - 
             \left( f(s) \wedge g(s)\right) ds\right) 
            \eqdist Y\left( \int_0^t f(s) ds\right),
\end{equation*}
where $Y$ is a unit rate Poisson process and we define
$a\wedge b \eqdef \min\{a,b\}$.   Thus, we are coupling the processes by splitting up the intensity functions into two pieces, one shared
\[
f(s) \wedge g(s),
\]
and the other not, and then using the same noise, $Y_1$, on the shared portion.

We return to the problem of coupling the main processes of interest to us.  For ease of notation, we will  couple the processes $X^{\theta+ \epsilon}$ and $X^{\theta}$ as opposed to $X^{\theta + \epsilon/2}$ and $X^{\theta-\epsilon/2}$, with the understanding that generating the centered difference is performed in the obvious manner.  
We generate our coupled processes ($X^{\theta+ \epsilon},X^{\theta})$ via:  
\begin{align}\label{eq:main_coupling}
\begin{split}
	&X^{\theta + \epsilon}(t) = X^{\theta + \epsilon}(0) + \sum_k Y_{k,1}\left( \int_0^t \lambda_k^{\theta + \epsilon}(X^{\theta + \epsilon}(s)) \wedge \lambda_k^{\theta}(X^{\theta}(s))ds\right) \zeta_k\\
	&\hspace{1.25in}+ \sum_k Y_{k,2}\left( \int_0^t \lambda_k^{\theta + \epsilon}(X^{\theta + \epsilon}(s)) - \lambda_k^{\theta + \epsilon}(X^{\theta + \epsilon}(s)) \wedge \lambda_k^{\theta}(X^{\theta}(s))ds\right) \zeta_k\\
	&X^{\theta}(t) = X^{\theta}(0) + \sum_k Y_{k,1}\left( \int_0^t \lambda_k^{\theta + \epsilon}(X^{\theta + \epsilon}(s)) \wedge \lambda_k^{\theta}(X^{\theta}(s)) ds\right) \zeta_k\\
	&\hspace{0.99in} + \sum_k Y_{k,3}\left( \int_0^t \lambda_k^{\theta}(X^{\theta}(s)) - \lambda_k^{\theta + \epsilon}(X^{\theta + \epsilon}(s)) \wedge \lambda_k^{\theta}(X^{\theta}(s))ds\right) \zeta_k,
	\end{split}
\end{align}
where the $Y_{k,i}$ are unit-rate Poisson processes and all other notation is as before.  Thus, and just as in the example pertaining to the non-homogeneous Poisson processes above, the effect of the intensity function $\lambda_k^{\theta + \epsilon}$ on the process $X^{\theta + \epsilon}$ has been split into two pieces: one of size $\lambda_k^{\theta + \epsilon}(X^{\theta + \epsilon}(s)) \wedge \lambda_k^{\theta}(X^{\theta}(s))ds$, and  one of size
\[
\lambda_k^{\theta + \epsilon}(X^{\theta + \epsilon}(s)) - \lambda_k^{\theta + \epsilon}(X^{\theta + \epsilon}(s)) \wedge \lambda_k^{\theta}(X^{\theta}(s))ds.
\]
Further, since the two processes $X^{\theta + \epsilon}$ and $X^{\theta}$ share the contribution of each of the terms with intensity 
\[
   \lambda_k^{\theta + \epsilon}(X^{\theta + \epsilon}(s)) \wedge \lambda_k^{\theta}(X^{\theta}(s))ds
\]
we expect them to be highly correlated.  It is important to note that the marginal processes have the same distributions as the respective processes generated via \eqref{eq:main_indexed}.  This fact can be seen by noting that \eqref{eq:main_coupling} is a continuous time Markov chain and that the transition rates of the marginal  processes are identical to those of \eqref{eq:main_indexed} with the corresponding rate constants.  Note also that the coupling \eqref{eq:main_coupling} is essentially the same as in the toy problems above where we coupled $Z_1$ and $Z_2$.  

A coupling similar to \eqref{eq:main_coupling} first appeared in \cite{Kurtz82}.  More recently, it was used in \cite{AndersonGangulyKurtz} to study the strong error of different approximation methods in the discrete stochastic case, and  in  \cite{AndersonHigham2011} to generate paths so as to apply multi-level Monte Carlo techniques in the continuous time Markov chain setting.  The application of the coupling \eqref{eq:main_coupling} towards the problem of parametric sensitivity analysis is  the main contribution of this paper.  

As discussed at the end of Section \ref{sec:intro},  the process $(X^{\theta + \epsilon},X^{\theta})$ satisfying \eqref{eq:main_coupling} is a continuous time Markov chain with state space $\Z^d \times \Z^d$.  Therefore, all analytical and computational techniques developed for the study of continuous time Markov chains will be applicable to this system, and, hence, to the problem of computing sensitivities.  

Before proceeding with the analysis, we give the algorithm for generating a path $(X^{\theta + \epsilon},X^{\theta})$ via \eqref{eq:main_coupling}.  We note that the method below is the next reaction method applied to \eqref{eq:main_coupling} \cite{Anderson2007a, Gibson2000}.  See \cite{Anderson2007a} for a thorough explanation of how the next reaction method is equivalent to simulating representations of the forms considered here.    Below, we will denote a uniform$[0,1]$ random variable by rand$(0,1)$, and we remind the reader that if $U \sim \text{rand}(0,1)$, then $\ln(1/U)$ is an exponential random variable with a parameter of one.  All random variables generated are assumed to be independent of each other and all previous random variables.  It is assumed that the processes start with the same initial condition, though this can be weakened in the obvious manner.  Finally, we note that it is also possible to simulate the continuous time Markov chain \eqref{eq:main_coupling} by the obvious adaption of Gillespie's direct, or optimized direct, algorithm.  While we do not formally provide that algorithm here, it will be problem specific as to which implementation (Gillespie versus next reaction method) is more efficient.  

\begin{algorithm}
   [Simulation of the representation \eqref{eq:main_coupling}]
\textbf{Initialize}.  Set $X^{\theta + \epsilon} = X^{\theta}  =x$ and $t = 0$.   For each $k$ and $i \in \{1,2,3\}$, set 
\begin{itemize}
	\item $P_{k,i} = \ln(1/u_{k,i})$, where $u_{k,i}$ is rand$(0,1)$.
	\item $T_{k,i} = 0$.
\end{itemize}  
Repeat the following steps:
  \begin{enumerate}[$(i)$]
  \item For each $k$, set
  	\begin{itemize}
  		\item $A_{k,1} = \lambda_k^{\theta + \epsilon}(X^{\theta + \epsilon}) \wedge \lambda_k^{\theta}(X^{\theta})$.
		\item $A_{k,2} =  \lambda_k^{\theta + \epsilon}(X^{\theta + \epsilon}) - A_{k,1}$.
		\item $A_{k,3} = \lambda_k^{\theta}(X^{\theta}) - A_{k,1}$.
  	\end{itemize}
\item  For each $k$ and $i \in \{1,2,3\}$, set 
\begin{equation*}
   \Delta t_{k,i} = \left\{ \begin{array}{cr}
   (P_{k,i} - T_{k,i})/A_{k,i}, & \text{ if } A_{k,i} > 0\\
   \infty, & \text{ if } A_{k,i} = 0
   \end{array} \right. .
 \end{equation*}
\item Set $\Delta = \min_{k,i}\{\Delta t_{k,i}\}$, and let $\mu \equiv \{k,i\}$ be the indices where the minimum is achieved.

\item Set $t = t + \Delta$.
\item Update state vectors according to reaction $\zeta_{\mu}$ (where minimum occurred in step $(iii)$):
\begin{align*}
	(X^{\theta + \epsilon},X^{\theta}) = \left\{ \begin{array}{cc}
		(X^{\theta + \epsilon},X^{\theta}) + (\zeta_k,\zeta_k), & \text{if } i = 1\\
		(X^{\theta + \epsilon},X^{\theta}) + (\zeta_k,0), & \text{if } i = 2\\
		(X^{\theta + \epsilon},X^{\theta}) + (0,\zeta_k), & \text{if } i = 3\\
	\end{array} \right. .
\end{align*}
\item For each $k$ and $i \in \{1,2,3\}$, set $T_{k,i} = T_{k,i} + A_{k,i}\times \Delta$.
\item Set $P_{\mu} = P_{\mu} + \ln(1/u)$, where $u$ is rand$(0,1)$.
\item Return to step $(i)$ or quit.
  \end{enumerate}
  \label{alg:NRM}
\end{algorithm}
 
Note that at most two of $A_{k,1}, A_{k,2}, A_{k,3}$ will be non-zero at each step.  Further, it will often be that $A_{k,1} \gg \max\{A_{k,2}, A_{k,3}\}$ and the processes will move together the vast majority of the time (which is, of course, the whole point of such a coupling), showing that the cost of generating the  path $(X^{\theta+\epsilon}, X^{\theta})$ will be  less than the cost of generating two paths via the representation \eqref{eq:main_indexed}.  This fact is observed in the data collected on the numerical examples in Section \ref{sec:examples}.

\subsubsection*{The Common Reaction Path method}
We now revisit the point that the strategy being proposed here is similar to the one proposed  in \cite{Khammash2010}, where instead of the coupling \eqref{eq:main_coupling} the authors used what is equivalent to
\begin{align}\label{eq:Rathinam}
\begin{split}
	X^{\theta + \epsilon}(t) &= X^{\theta + \epsilon}(0) + \sum_k Y_{k}\left( \int_0^t \lambda_k^{\theta + \epsilon}(X^{\theta + \epsilon}(s))ds\right) \zeta_k\\
	X^{\theta}(t) &= X^{\theta}(0) + \sum_k Y_{k}\left( \int_0^t  \lambda_k^{\theta}(X^{\theta}(s)) ds\right) \zeta_k,
	\end{split}
\end{align}
where the $Y_k$ are independent unit-rate Poisson processes and all other notation is as before.  The key point is that they are using the same Poisson processes for the generation of each path.  As stated in Section \ref{sec:intro}, the estimator built with paths generated via \eqref{eq:Rathinam} is quite capable in many circumstances.  The main differences between their method and the one being proposed in this paper via \eqref{eq:main_coupling} are:
\begin{enumerate}
	\item The processes generated via \eqref{eq:main_coupling} are generally coupled tighter than those generated via \eqref{eq:Rathinam}, resulting in a lower variance for the estimator, sometimes substantially so.  However, sometimes the coupling \eqref{eq:Rathinam} produces a lower variance estimator than \eqref{eq:main_coupling} when the terminal time $T$ is small.  These facts will be demonstrated via example in Section \ref{sec:examples} and discussed more below.
	\item The model \eqref{eq:main_coupling} is more amenable to analysis as the centered counting processes of \eqref{eq:main_coupling} are martingales with respect to the natural filtration \cite{AndersonGangulyKurtz}, which is not the case for \eqref{eq:Rathinam}.  
  
	\item Implementation of \eqref{eq:main_coupling} is simpler than that of \eqref{eq:Rathinam} as \eqref{eq:main_coupling} does not require the generation of many independent seeds for the pseudo-random number generator.  In fact, simulation of \eqref{eq:main_coupling} is no more challenging than simulating any continuous time Markov chain.
	\item The coupling \eqref{eq:main_coupling} makes the problem of computing the difference between two paths into one of computing a single path of a different continuous time Markov chain with an enlarged state space.  
\end{enumerate}

The following example is chosen to highlight the advantages of the coupling \eqref{eq:main_coupling} over that of \eqref{eq:Rathinam}.

\begin{example}\label{ex:mRNA}
Consider the simple model in which an mRNA molecule is created and degraded
\begin{equation}
	\emptyset \overset{\theta}{\underset{0.1}{\rightleftarrows}} M,
	\label{simple}
\end{equation}
which is equivalent to an $M/M/\infty$ queue with arrival rate $\theta$ and service rate $0.1$.  Here we are using the common convention of putting the rate constant of a reaction next to the corresponding reaction vector.
We suppose that  we want to understand the sensitivity of the expected number of mRNA molecules with respect to the parameter $\theta\approx 2$.   We  consider how the different representations \eqref{eq:main_coupling} and \eqref{eq:Rathinam} ``should'' behave on this model, whose representation via \eqref{eq:main_indexed} is
\begin{equation}\label{eq:mRNA}
	X^\theta(t) = X(0) + Y_1\left( \theta t\right) - Y_2\left( \int_0^t 0.1 X^\theta(s) ds\right).
\end{equation}
For $\epsilon>0$, let  $(X^{\theta + \epsilon},X^\theta)$ satisfy \eqref{eq:main_coupling}, which for this example is
\begin{align*}
	X^{\theta+\epsilon}(t) &= X^{\theta+\epsilon}(0) + Y_{1,1} (\theta t)+ Y_{1,2}(\epsilon t) \\
	&\hspace{.4in}  -  Y_{2,1}\left( \int_0^t 0.1 X^\theta(s) ds\right) - Y_{2,2}\left(\int_0^t 0.1 (X^{\theta + \epsilon}(s) - X^\theta(s)) ds\right)\\
	X^{\theta}(t) &= X^{\theta}(0) + Y_{1,1} (\theta t) -  Y_{2,1}\left( \int_0^t 0.1 X^\theta(s) ds\right), 
\end{align*}
where we have used that with this coupling $X^{\theta+\epsilon}(t) \ge X^\theta(t)$, for all $t \ge 0$ if $\epsilon \ge 0$.  
Therefore, assuming that $X^{\theta + \epsilon}(0)=X^\theta(0)$, we have
\begin{align*}
	X^{\theta + \epsilon}(t) - X^\theta(t) = Y_{1,2}(\epsilon t)  -  Y_{2,2}\left(\int_0^t 0.1 (X^{\theta + \epsilon}(s) - X^\theta(s)) ds\right).
\end{align*}
Setting $Z^{\theta,\epsilon} = X^{\theta + \epsilon} - X^\theta$, we see that $Z^{\theta,\epsilon}$ itself can be viewed as the solution to \eqref{eq:mRNA}, though with zero initial condition and input rate $\epsilon$.  The mean and variance can be solved for as functions of time and satisfy
\begin{align}
	\E Z^{\theta,\epsilon}(t) &= \E (X^{\theta + \epsilon}(t) - X^{\theta}(t)) =  \frac{\epsilon}{0.1}(1 - e^{-0.1 t}) \label{eq:mean11}\\
	\mathsf{Var}(Z^{\theta,\epsilon}(t)) &= \mathsf{Var}(X^{\theta + \epsilon}(t) - X^{\theta}(t))= \frac{\epsilon}{0.1}(1 - e^{-0.1t}).\label{eq:var11}
\end{align}

On the other hand, if $(X^{\theta + \epsilon},X^\theta)$ satisfy the coupling \eqref{eq:Rathinam}, then
\begin{align*}
	X^{\theta+\epsilon}(t) &= X^{\theta+\epsilon}(0) + Y_{1} (\theta t + \epsilon t)  -  Y_{2}\left( \int_0^t 0.1 X^{\theta +\epsilon}(s) ds\right)\\
	X^{\theta}(t) &= X^{\theta}(0) + Y_{1} (\theta t) -  Y_{2}\left( \int_0^t 0.1 X^\theta (s) ds\right), 
\end{align*}
and 
\begin{equation}
	X^{\theta+\epsilon}(t) - X^{\theta}(t) = Y_{1} (\theta t + \epsilon t)  -Y_{1}(\theta t) - \left[ Y_{2}\left( \int_0^t 0.1 X^{\theta +\epsilon}(s) ds\right)  -  Y_{2}\left( \int_0^t 0.1 X^\theta (s) ds\right)\right].\label{diff_bad}
\end{equation}
In this case, we still have that $\E [X^{k + \epsilon}(t) -X^k(t)]$ satisfies the right hand side of \eqref{eq:mean11}.  However, the variance can not be calculated with such ease as for \eqref{eq:var11}.  We note, however, for large $t$, we will have  $\theta t +\epsilon t \gg \theta t$, and therefore anticipate 
\begin{equation*}
	\int_0^t 0.1X^{\theta+\epsilon}(s) ds \gg \int_0^t 0.1 X^\theta(s)ds,
\end{equation*}
implying that the two processes $X^{\theta + \epsilon}$ and $X^{\theta}$ should decouple, and behave independently.  This is demonstrated in a numerical example in Section \ref{sec:examples} where we show the variance of the difference \eqref{diff_bad} converges to 40, which is the same as if $X^{\theta + \epsilon}$ and $X^{\theta}$ were generated independently, and {\em substantially} larger than the bound given in \eqref{eq:var11} for small $\epsilon$.  

As will be discussed immediately below, we also expect to see this ``decoupling'' when  Gillespie's algorithm is implemented with common random numbers (CRN).  This also will be demonstrated by example in Section \ref{sec:examples}. \hfill $\square$
\end{example}

The above example gives a heuristic as to why the coupling \eqref{eq:main_coupling} will tend to give a  lower variance than \eqref{eq:Rathinam}.  
For processes generated by \eqref{eq:main_coupling}, whenever $X^{\theta + \epsilon}(t) \approx X^{\theta}(t)$ during the course of the simulation the processes have  re-coupled, regardless of the the history of the process up to that time.  On the other hand, if $X^{\theta + \epsilon}$ and $X^{\theta}$ are generated via \eqref{eq:Rathinam}, then $X^{\theta + \epsilon} \approx X^{\theta}$ need not imply
\[
	\int_0^t \lambda_k^{\theta + \epsilon}(X^{\theta + \epsilon}(s)) ds  \approx \int_0^t \lambda_k^{\theta + \epsilon}(X^{\theta}(s) ds),
\]
and so the two processes could be exploring completely different portions of the Poisson processes.  Thus, even when $X^{\theta + \epsilon}(t) = X^{\theta}(t)$ in the course of the simulation of \eqref{eq:Rathinam}, the integrated intensities will not be equal and so the processes are no longer coupled as tightly as they were at time zero.  As time increases, this problem could get worse and the processes can decouple completely (as happens in the example above).

\subsubsection*{Common Random Numbers and Gillespie's algorithm}
The standard method of using common random numbers in the implementation of Gillespie's algorithm will suffer the same defect as \eqref{eq:Rathinam} in that for large times the coupled processes can decouple.  To understand why, we need to give the correct representation for Gillespie's algorithm (which is equivalent to simulating the embedded discrete time Markov chain).  The following representation can be found in \cite{AndKurtz2011}.  We define 
\[
  \lambda_0(x)=\sum_k \lambda_k(x),\quad \text{ and } \quad q_k(x)=\sum_{i=1}^k\lambda_
  i(x)/\lambda_0(x).
\]
Let $Y$ be a unit rate Poisson process and let $\{\xi_i\}$ be an i.i.d.  sequence of uniform(0,1) random variables that are also independent of $Y$.  Then let $X$ satisfy 
\begin{align}\label{eq:sim_Gill}
\begin{split}
  R_0(t )&= Y\left(\int_0^t\lambda_0(X(s))ds\right)\\
  X(t) &= X(0)+\sum_k \zeta_k \int_0^t{\bf 1}_{(q_{k-1}(X(s-)),q_k(X(s
    -)]}(\xi_{R_0(s-)})dR_0(s).
    \end{split}
\end{align}
The counting process $R_0$ is determining the jump times, which are seen to be exponential random variables with parameter $\lambda_0(X(s-))$.  The uniform random variables are then used to select which reaction occurs, with the $k$th reaction being chosen with probability $\lambda_k(X(s-))/\lambda_0(X(s-))$.
Simulation of \eqref{eq:sim_Gill} is called Gillespie's algorithm (or simply: simulating the embedded discrete time Markov chain).
The standard common random number + Gillespie algorithm finite difference method, which is probably the most common coupling method used today in the context of finite differences, then  consists of using the same 
$Y$ and choice of $\{\xi_i\}$ for the construction of $X^{\theta}$ and $X^{\theta + \epsilon}$, and will decouple for the same reasons as that of \eqref{eq:Rathinam}.  This is demonstrated numerically in an example in Section \ref{sec:examples}.

\subsubsection*{Naive couplings: a cautionary tale}

At this point it may be tempting to try to couple the processes generated via \eqref{eq:main_coupling} even tighter by using the same Poisson processes for each of $Y_{k,2}$ and $Y_{k,3}$.  This would, in effect, be a highbred version of the Common Reaction Path and Coupled Finite Difference methods.  That is, one may be tempted to use
\begin{align}\label{eq:bad_coupling}
\begin{split}
	&X^{\theta + \epsilon}(t) = X^{\theta + \epsilon}(0) + \sum_k Y_{k,1}\left( \int_0^t \lambda_k^{\theta + \epsilon}(X^{\theta + \epsilon}(s)) \wedge \lambda_k^{\theta}(X^{\theta}(s))ds\right) \zeta_k\\
	&\hspace{.7in}+ \sum_k Y_{k,2}\left( \int_0^t \lambda_k^{\theta + \epsilon}(X^{\theta + \epsilon}(s)) - \lambda_k^{\theta + \epsilon}(X^{\theta + \epsilon}(s)) \wedge \lambda_k^{\theta}(X^{\theta}(s))ds\right) \zeta_k\\
	&X^{\theta}(t) = X^{\theta}(0) + \sum_k Y_{k,1}\left( \int_0^t \lambda_k^{\theta + \epsilon}(X^{\theta + \epsilon}(s)) \wedge \lambda_k^{\theta}(X^{\theta}(s)) ds\right) \zeta_k\\
	&\hspace{.7in} + \sum_k Y_{k,2}\left( \int_0^t \lambda_k^{\theta}(X^{\theta}(s)) - \lambda_k^{\theta + \epsilon}(X^{\theta + \epsilon}(s)) \wedge \lambda_k^{\theta}(X^{\theta}(s))ds\right) \zeta_k,
	\end{split}
\end{align}
where we are now using the same Poisson process for all of the auxiliary processes.
In fact, this does not work:  the marginal distributions of $(X^{\theta+\epsilon},X^{\theta})$ as generated by \eqref{eq:bad_coupling} are not the same as the original processes and so this coupling should \textit{not} be used. In fact, the marginal distributions can be so different that the coupled processes will converge to the wrong value as $\epsilon \to 0$.  This fact is best demonstrated by an example.

\begin{example}\label{ex:bad_example}
	Consider the system arising from the single reaction
	\[
		X \overset{\theta}{\to} \emptyset,
	\]
	with $X_0=1$. The stochastic equation of the form \eqref{eq:main} governing this system is
	\[
		X^{\theta}(t) = 1 - Y\left(\theta \int_0^t X^{\theta}(s) ds\right). 
	\]
	Hence, the process can only take the values  of one and zero and $\E X^{\theta}(1) = \exp\{-\theta\}$, which implies $\displaystyle \frac{d}{d\theta} \E X^{\theta}(1) = - \exp\{-\theta\}$.  In particular, 
	\[
		 \frac{d}{d\theta} \E X^{\theta}(1)\bigg|_{\theta = 1} = - e^{-1}.
	\]
	For the coupling \eqref{eq:bad_coupling} for this model we have
	\begin{align*}
		\epsilon^{-1} \E [X^{1 + \epsilon/2} - X^{1-\epsilon/2}] = -e^{-1} \frac{2}{2+\epsilon} \left( e^{\epsilon/2} - e^{-\epsilon/2}\right) = -e^{-1} \epsilon + O(\epsilon^2),
	\end{align*}
	which converges to zero as $\epsilon \to 0$. 
	Perhaps the simplest way to compute the above expectation is to find the condition on the first jump times of the underlying Poisson processes that guarantee the random variable $X^{1 + \epsilon/2} - X^{1-\epsilon/2}$ takes a value of negative one.  This event has a probability of $O(\epsilon^2)$, implying the result. This computation is left to the interested reader.\hfill $\square$
\end{example}

\subsection{Analytical results}
\label{sec:analytical_results}

The following theorem is the main analytical result of this paper and allows us to conclude that for any function $f$ satisfying the assumptions of Theorem \ref{thm:main} and $(X^{\theta + \epsilon},X^{\theta})$ satisfying \eqref{eq:main_coupling}
\begin{align*}
	\mathsf{Var}\left(f(X^{\theta+ \epsilon}(t)) - f(X^{\theta}(t))\right) \le C_{t,f,M} \epsilon,
\end{align*}
for some $C_{t,f,M}>0$ depending upon $t$, $f$, and $M$ (the number of reactions).

\begin{theorem}\label{thm:main}
	Suppose $(X^{\theta + \epsilon},X^{\theta})$ satisfy \eqref{eq:main_coupling} with our running Assumptions \ref{assump:GL} and \ref{assump:theta}. Let $f:\R^d \to \R$ be a $C^1$ function with bounded first derivative on all $x \in \mathcal S$.  Then, for any $T>0$ there is a $C_{T,f,M} > 0$ for which
	\begin{equation*}
		 \E \sup_{t\le T} \left( f(X ^{\theta + \epsilon} (t)) - f(X^{\theta}(t))\right)^2 \le C_{T,f,M} \epsilon.
	\end{equation*}
\end{theorem}

We provide two Lemmas, giving the $L^1$ and $L^2$ bound on the difference between $X^{\theta + \epsilon}$ and $X^{\theta}$, before proving Theorem \ref{thm:main}.

\begin{lemma}\label{lem:L1}
	Suppose $(X^{\theta + \epsilon},X^{\theta})$ satisfy \eqref{eq:main_coupling} with our running Assumptions \ref{assump:GL} and \ref{assump:theta}. Then, for $T>0$ there is a $C_{T,M} > 0$ for which
	\begin{equation*}
		\E \sup_{t \le T} \left| X ^{\theta + \epsilon} (t) - X^{\theta}(t)\right| \le C_{T,M} \epsilon.
	\end{equation*}
\end{lemma}

\begin{proof}
	Let $T>0$. For any $s \ge 0$
	\begin{align}
		X^{\theta + \epsilon}(s) - X^{\theta}(s) &= M^{\theta,\epsilon}(s) + \int_0^s F^{\theta + \epsilon}(X^{\theta + \epsilon}(r)) - F^{\theta}(X^{\theta}(r)) dr,
		\label{eq:main_diff2}
	\end{align}
where $M^{\theta,\epsilon}$ is a martingale with quadratic covariation
\begin{align*}
	[M^{\theta,\epsilon}]_t &= \sum_k \left(N^{\theta,\epsilon}_{k,2}(t) + N^{\theta,\epsilon}_{k,3}(t)\right) \zeta_k \zeta_k^T,
\end{align*}
where
\begin{align*}
	N^{\theta,\epsilon}_{k,2}(t) &\eqdef Y_{k,2} \left( \int_0^t \lambda_k^{\theta + \epsilon}(X^{\theta + \epsilon}(s)) - \lambda_k^{\theta + \epsilon}(X^{\theta + \epsilon}(s)) \wedge  \lambda_k^{\theta}(X^{\theta}(s))ds \right)\\
	N^{\theta,\epsilon}_{k,3}(t) &\eqdef Y_{k,3}\left( \int_0^t \lambda_k^{\theta}(X^{\theta}(s)) - \lambda_k^{\theta + \epsilon}(X^{\theta + \epsilon}(s)) \wedge  \lambda_k^{\theta}(X^{\theta}(s))ds \right).
\end{align*}
Therefore, for $s \le t$
\begin{align*}
	|X^{\theta + \epsilon}(s) - X^{\theta}(s)| &\le \sup_{r \le t} |M^{\theta,\epsilon}(r)| + \int_0^s |F^{\theta + \epsilon}(X^{\theta + \epsilon}(r)) - F^{\theta}(X^{\theta + \epsilon}(r)) | dr\\
	&\hspace{.6in} + \int_0^s |F^{\theta}(X^{\theta + \epsilon}(r)) - F^{\theta}(X^{\theta}(r)) | dr\\
	&\le  \sup_{r \le t} |M^{\theta,\epsilon}(r)| + K_2 \epsilon s + K_1 \int_0^s |X^{\theta + \epsilon}(r) - X^{\theta}(r) | dr\\
	&\le  \sup_{r \le t} |M^{\theta,\epsilon}(r)| + K_2 \epsilon t + K_1 \int_0^t\sup_{u \le r} |X^{\theta + \epsilon}(u) - X^{\theta}(u) | dr,
\end{align*}
where $K_1,K_2$ are the constants of Assumption \ref{assump:GL} and \ref{assump:theta}, respectively.
As the above inequality holds for all $s\le t$, we have
\begin{equation}\label{eq:L1_exp}
\sup_{s\le t} |X^{\theta + \epsilon}(s) - X^{\theta}(s)| \le   \sup_{r \le t} |M^{\theta,\epsilon}(r)| + K_2 \epsilon t + K_1 \int_0^t\sup_{u \le r} |X^{\theta + \epsilon}(u) - X^{\theta}(u) | dr.
\end{equation}
By the Burkholder-Davis-Gundy inequality and the fact that $\sqrt{z} \le z$ for all nonnegative integers, we have the existence of a $C_2>0$ for which
\begin{align}\label{eq:need_L1}
\begin{split}
	\E& \sup_{r \le t} |M^{\theta,\epsilon}(r)| \le C_2 \sum_k \int_0^t \E | \lambda_k^{\theta+ \epsilon}(X^{\theta +\epsilon}(u))  -  \lambda_k^{\theta}(X^{\theta }(u)) | du\\
	&\le C_2 \sum_k\bigg[ \int_0^t \E | \lambda_k^{\theta+ \epsilon}(X^{\theta +\epsilon}(u))  -  \lambda_k^{\theta}(X^{\theta +\epsilon}(u)) | du + \int_0^t \E |  \lambda_k^{\theta}(X^{\theta +\epsilon}(u))   -  \lambda_k^{\theta}(X^{\theta }(u)) | du\bigg]\\
	&\le C_3 \epsilon t + C_4 \int_0^t \E |X^{\theta +\epsilon}(u)  -  X^{\theta }(u) | du\\
	& \le  C_3 \epsilon t + C_4 \int_0^t\E \sup_{u \le r} |X^{\theta +\epsilon}(u)  -  X^{\theta }(u) | dr,
	\end{split}
\end{align}
where $C_3$ and $C_4$ are constants independent of $\epsilon$ or $T$, and depend linearly on $M$ (the number of reactions).
Taking expectations of \eqref{eq:L1_exp}, applying \eqref{eq:need_L1}, and using Gronwall's inequality gives the desired result.

\end{proof}

\begin{lemma}\label{lem:L2}
	Suppose $(X^{\theta + \epsilon},X^{\theta})$ satisfy \eqref{eq:main_coupling} with our running Assumptions \ref{assump:GL}  and \ref{assump:theta}. Then, for $T>0$ there is a $C_{T,M} > 0$ for which
	\begin{equation*}
		\E \sup_{t\le T} \left| X ^{\theta + \epsilon} (t) - X^{\theta}(t)\right|^2 \le C_{T,M} \epsilon.
	\end{equation*}
\end{lemma}

\begin{proof}
  Returning to \eqref{eq:main_diff2} in the proof of Lemma \ref{lem:L1}, we have that
  \begin{align*}
  	|X^{\theta + \epsilon}(s) - X^{\theta}(s)|^2 &\le  2|M^{\theta,\epsilon}(s)|^2 + 2T \int_0^s |F^{\theta + \epsilon}(X^{\theta + \epsilon}(r)) - F^{\theta}(X^{\theta}(r)) |^2 dr.
  \end{align*}
  The proof is now essentially the same as that of Lemma \ref{lem:L1}, though  \eqref{eq:need_L1} in combination with Lemma \ref{lem:L1} is used to bound the martingale term.
\end{proof}

Note that the expected difference of the $p$th moment, for any $p \ge 1$, can be estimated in the same manner as above.

\begin{example}
	To demonstrate the above Lemmas, we consider the following simple example:
	\begin{equation*}
		\emptyset \overset{\theta}{\to} S,
	\end{equation*}
	where the parameter of interest is the rate constant $\theta$.  For this example we have that $\lambda^{\theta}(x) \equiv \theta$, and so
	\begin{align*}
		X^{\theta + \epsilon}(t) &= Y_1( \theta t) + Y_2(\epsilon t)\\
		X^{\theta}(t) &= Y_1( \theta t).
	\end{align*}
	Hence $X^{\theta+ \epsilon}(t) - X^{\theta}(t) = Y_2(\epsilon t)$, and the statements of Lemmas \ref{lem:L1} and \ref{lem:L2} follow immediately.  Further, the lemmas are shown to be sharp.
\end{example}

\begin{proof}(of Theorem \ref{thm:main}.)
	By Taylor's theorem combined with our assumption on $f$,  we have that for some $C_f>0$
	\begin{align*}
		  \left( f(X ^{\theta + \epsilon} (t)) - f(X^{\theta}(t))\right)^2 &\le C_f |X^{\theta + \epsilon}(t) - X^{\theta}(t)|^2,
	\end{align*}
	and the result follows by application of Lemma \ref{lem:L2}.
\end{proof}

We return now to \eqref{eq:difference} and \eqref{eq:d} and note that with $X^{\theta+\epsilon/2},X^{\theta-\epsilon/2}$ generated via the coupling proposed here,
\begin{align*}
	\mathsf{Var}(d_{[i]}(\epsilon)) &= O(\epsilon^{-1})\\
	\mathsf{Var}(D_R(\epsilon)) &= O(R^{-1}\epsilon^{-1}),
\end{align*}
which are an order of magnitude lower, in terms of $\epsilon$, than the respective variances when the processes are generated independently.  As discussed at the end of Section \ref{sec:theWhy}, this fact leads to a decrease in the computational work (and simulation time) required to solve a given problem to a desired tolerance by a factor of $\epsilon$, yielding potentially enormous savings.

\section{Numerical examples}
\label{sec:examples}

We compare our method with existing methods on three different models: a basic model for the production of mRNA and proteins, an $M/M/\infty$ queue, and a genetic toggle switch from \cite{Khammash2010}.  Because the common reaction path method of \cite{Khammash2010} tends to perform at least as well as the usual implementation of common random numbers with Gillespie's algorithm, we choose to only include the Common Reaction Path method in our comparison (except for one plot in Numerical Example 2 to demonstrate the decoupling alluded to at the end of Section \ref{sec:results}).

\begin{num_ex}

Consider the  model of gene transcription and translation
 \begin{equation}\label{ex1}
 	G \overset{2}{\to} G + M, \quad M \overset{10}{\to} M + P, \quad M \overset{\theta}{\to} \emptyset, \quad P \overset{1}{\to} \emptyset,
 \end{equation}
 where a single gene is being translated into mRNA, which is then being transcribed into proteins. The final two reactions represent degradation of the mRNA and protein molecules, respectively.  Assuming that there is a single gene copy, the stochastic equation \eqref{eq:main} for this model is
 \begin{align}
 \begin{split}
 	X^{\theta}(t) &= X^{\theta}(0) + Y_1(2t) \left( \begin{array}{cc}
		1\\
		0
	\end{array} \right) + Y_2 \left( \int_0^t 10 X^{\theta}_1(s) ds\right)\left( \begin{array}{cc}
		0\\
		1
	\end{array} \right) + Y_3 \left( \int_0^t \theta X^{\theta}_1(s) ds\right)\left( \begin{array}{cc}
		-1\\
		0
	\end{array} \right)\\
	&\hspace{.2in} + Y_4 \left( \int_0^t  X^{\theta}_2(s) ds\right)\left( \begin{array}{cc}
		0\\
		-1
	\end{array} \right),
	\end{split}
	\label{ex:Gene}
 \end{align}  
 where $X^{\theta}_1(t)$ and $X^{\theta}_2(t)$ give the number of mRNA and protein molecules at time $t$, respectively, and $Y_1,Y_2$ are independent unit-rate Poisson processes.
 Suppose the rate constant $\theta$ is of interest to us and we believe that $\theta \approx 1/4$.  We would like to estimate the sensitivity of the mean number of protein molecules at time $T = 30$, say, with respect to the parameter $\theta \approx 1/4$.  Here, it is a straightforward calculation to find that
 \begin{align*}
 	\E X_2^\theta(30)\bigg|_{\theta = 1/4} &\approx 79.941\qquad \text{and} \qquad 	\frac{d}{d\theta} \E X_2^{\theta}(30) \bigg|_{\theta = 1/4} \approx -318.073,
 \end{align*}
 if the initial condition is $X^{\theta}(0) = [0, 0]^T$.
 Defining 
  \begin{equation*}
 	J(\theta) \eqdef \frac{d}{d \theta} \E \left[ X^\theta_2(30)\right],
 \end{equation*}
 our goal is to efficiently estimate $J(1/4)$ and we compare the following methods on this problem:
 \begin{enumerate}[$(i)$]
 \item the usual crude Monte Carlo (CMC) estimator with independent samples, 
 \item the common reaction path (CRP) method of \cite{Khammash2010} using the coupled processes \eqref{eq:Rathinam},
 \item the coupled finite difference  (CFD) method being proposed in this paper using the coupling \eqref{eq:main_coupling},
 \item a Girsanov transformation method of Plyasunov and Arkin detailed in \cite{Plyasunov2007}.
 \end{enumerate}
    For all simulations, we assume an initial condition of zero mRNA and zero protein molecules.  
     We will denote the number of sample paths used in the construction of the relevant estimators \eqref{eq:difference} via $R$ and the perturbation in the centered finite difference via $\epsilon$.  
      
      \begin{table}
 \begin{center}
 \begin{tabular}{|c|c|c|c|c||c|}\hline
 	Method & $R$& $\epsilon = 1/20$  & $\epsilon =1/100$  & \# updates  & CPU time\\ \hline \hline
	CMC  & 1,000 & -276.2 $\pm \ 46.3$   & -472.7.1 $\pm\ 237.3$ &    $\approx 8.4\times 10^6$ & $\approx$ 9.6 S\\ \hline 
	CRP  & 1,000 & -323.1 $\pm \ 18.4$  & -321.0 $\pm \ 60.2$ & $\approx 8.4 \times 10^6$ & $\approx$ 10.1 S  \\ \hline
	CFD  & 1,000 & -323.7  $\pm \ 8.7$  &  -333.8 $\pm \ 28.0$ & $\approx 4.4 \times 10^6$ &  $\approx$ 6.5 S   \\ \hline   \hline 
	CMC  & 10,000 & -324.8 $\pm \ 14.7$ &  -305.4 $\pm \ 74.3$ &  $\approx 8.4 \times 10^7$ & $\approx$ 98.8 S \\ \hline
	CRP  & 10,000 & -325.5 $\pm \ 5.8$ & -328.6 $\pm \ 18.6$  & $\approx 8.4 \times 10^7$ & $\approx$ 105.4 S \\ \hline 
	CFD  & 10,000 & -320.0 $\pm \ 2.8$  & -316.6 $\pm \ 8.9$ &  $\approx 4.4 \times 10^7$ & $\approx$ 64.9 S \\ \hline \hline
	CMC & 40,000&  -322.7 $\pm \ 7.5$ & -341.9 $\pm \  37.3$ & $\approx 3.4 \times 10^8$ & $\approx$ 395.3 S\\ \hline
	CRP  & 40,000 & -319.6 $\pm \ 2.9$ & -310.6 $\pm \ 9.3$ & $\approx 3.4 \times 10^8$ & $\approx $ 411.5 S  \\ \hline 
	CFD & 40,000 & -321.6 $\pm \ 1.4$   & 317.8 $\pm\  4.4$ & $\approx 1.8 \times 10^8$    & $\approx$ 263.3 S \\ \hline
 \end{tabular}
  \caption{95\% confidence intervals and computational complexity for $(i)$ crude Monte Carlo (CMC), $(ii)$ the common reaction path (CRP) method of \cite{Khammash2010}, and $(iii)$ the coupled finite difference method (CFD) proposed here, applied to \eqref{ex1} in order to approximate $J(1/4)$ for different choices of $R$ and $\epsilon$. The exact value is $J(1/4) =  -318.073$.  Note that the bias of the centered finite difference is apparent when $\epsilon = 1/20$ and $R = $ 40,000.}
 \label{table1}
 \end{center}
\end{table}

     In Table \ref{table1}, we provide the 95\% confidence intervals computed using the crude Monte Carlo method with independent paths (CMC), the common reaction path method (CRP), and the coupled finite difference method (CFD) for different choices of  $R$, the number of paths simulated, and $\epsilon$, the perturbation of $\theta$.  For each method and choice of $R$, we also provide: $(i)$  an approximate total number of steps (and random numbers used) over the course of the entire simulation, and $(ii)$ an approximation of the CPU time required.  These numbers, which quantify the computational work required from each method, are essentially independent of $\epsilon$, and the numbers provided are the average of the actual values for the two different values of $\epsilon$ for a given method and choice of $R$.  These numbers should be used as a reference for the computational complexity required by the different methods with the understanding that CPU time will depend greatly upon implementation (the author used Matlab for all computations, which were performed on an Apple machine with a 2.2 GHz Intel i7 processor).  In Table \ref{table2}, we provide similar data for the Girsanov transformation method of \cite{Plyasunov2007}.

 \begin{table}
 \begin{center}
	\begin{tabular}{|c|c|c|c|}\hline 
		$R$ & Approximation & \# updates & CPU Time \\ \hline \hline 
		1,000 & -441.5 $\pm \ 156.5$  & $4.2 \times 10^6$ & 5.2 S  \\ \hline
		10,000 & -324.2 $\pm \ 49.2$  & $4.2 \times 10^7$ & 54.5 S \\ \hline
		40,000 &-327.8 $\pm  25.1$ & $1.7 \times 10^8$ & 207.6 S\\ \hline
	\end{tabular}
 \caption{95\% confidence intervals and computational complexity for the Girsanov transformation method of \cite{Plyasunov2007} applied to \eqref{ex1} in order to approximate $J(1/4)$. The exact value is $J(1/4) =  -318.073$.}
 \label{table2}
 \end{center}
\end{table}

While Tables \ref{table1} and \ref{table2} demonstrate that the method being proposed here can produce a more accurate estimate in less CPU time than the other methods, a more important statistic is the CPU time needed for each method to achieve a desired tolerance.  Therefore, we applied each method until the 95\% confidence interval was $\pm \ 6.0$.  The finite difference methods were applied with a perturbation size of $\epsilon = 1/40$.  Table \ref{table3} provides the number of updates needed by each method combined with the CPU time needed on our machine. We see that the coupled finite difference (CFP) method was  approximately 9 times more efficient that the common reaction path method,  and vastly more  efficient than both the Girsanov transformation method and the crude Monte Carlo with independent samples.

\begin{table}
\begin{center}
	\begin{tabular}{|c|c|c|c|c|}\hline 
		Method & $R$ & Approximation & \# updates & CPU Time \\ \hline \hline 
		Girsanov & 689,600 & -312.1 $\pm \ 6.0$    & $2.9 \times 10^{9}$   & 3,506.6 S \\ \hline 
		CMC & 246,000 & -319.3 $\pm \ 6.0$  & $2.1 \times 10^9$ & 2,364.8 S\\ \hline
		CRP  & 25,980  & -316.7 $\pm\ 6.0$ & $2.2 \times 10^8$ & 270.9 S \\ \hline
		CFD  & 4,580 &-319.9 $\pm \ 6.0$  & $2.0 \times 10^7$ & 29.2 S \\ \hline
	\end{tabular}
\end{center}
\caption{Required $R$, \# updates, and CPU time for each method to provide a 95\% confidence of $\pm\ 6.0$.  Each finite difference method used $\epsilon = 1/40$. The exact value is $J(1/4) =  -318.073$.}
\label{table3}
\end{table}

 Next, we simulated the system \eqref{ex:Gene} 5,000 times using each of the different methods and plotted the variance of the estimators versus time up to $T = 60$, see Figure \ref{fig:VarTime1}.    The finite difference methods were computed with a perturbation of size $\epsilon = 1/40$.   We note that the variance of each of the finite difference methods appears to converge, though the limiting value for the coupled finite difference method being proposed here converges to a value that is approximately 6.5 times lower than that of the common reaction path method, and 52 times lower than crude Monte Carlo.   Also note that the variance of the Girsanov transform method grows linearly in time, as expected, and is quite large for even moderate values of time, $t$.  
   \begin{figure}
    \begin{center}
    \begin{multicols}{2}\phantom{.}\vspace{.4in}
    \subfloat[Coupled Finite Differences and Common Reaction Path]{\label{fig:CFD_CRP_gene}
      \includegraphics[height=3.5in]{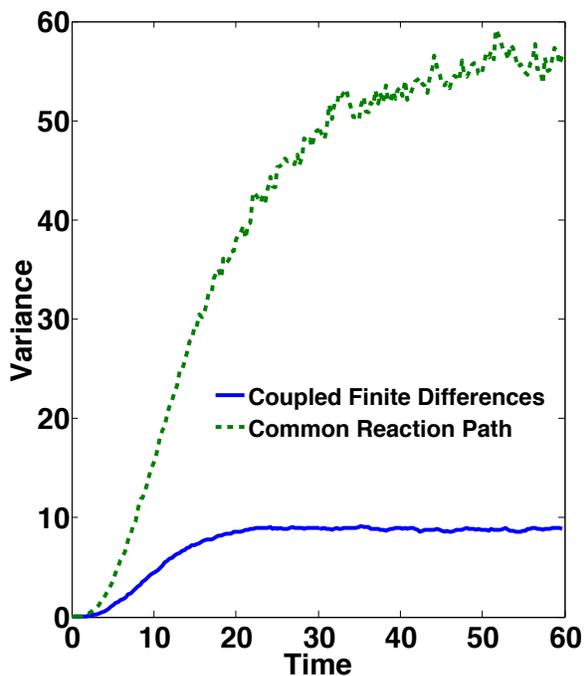}} \qquad
      \subfloat[Crude Monte Carlo] {\label{fig:cmc_gene}
        \includegraphics[height=2.7in]{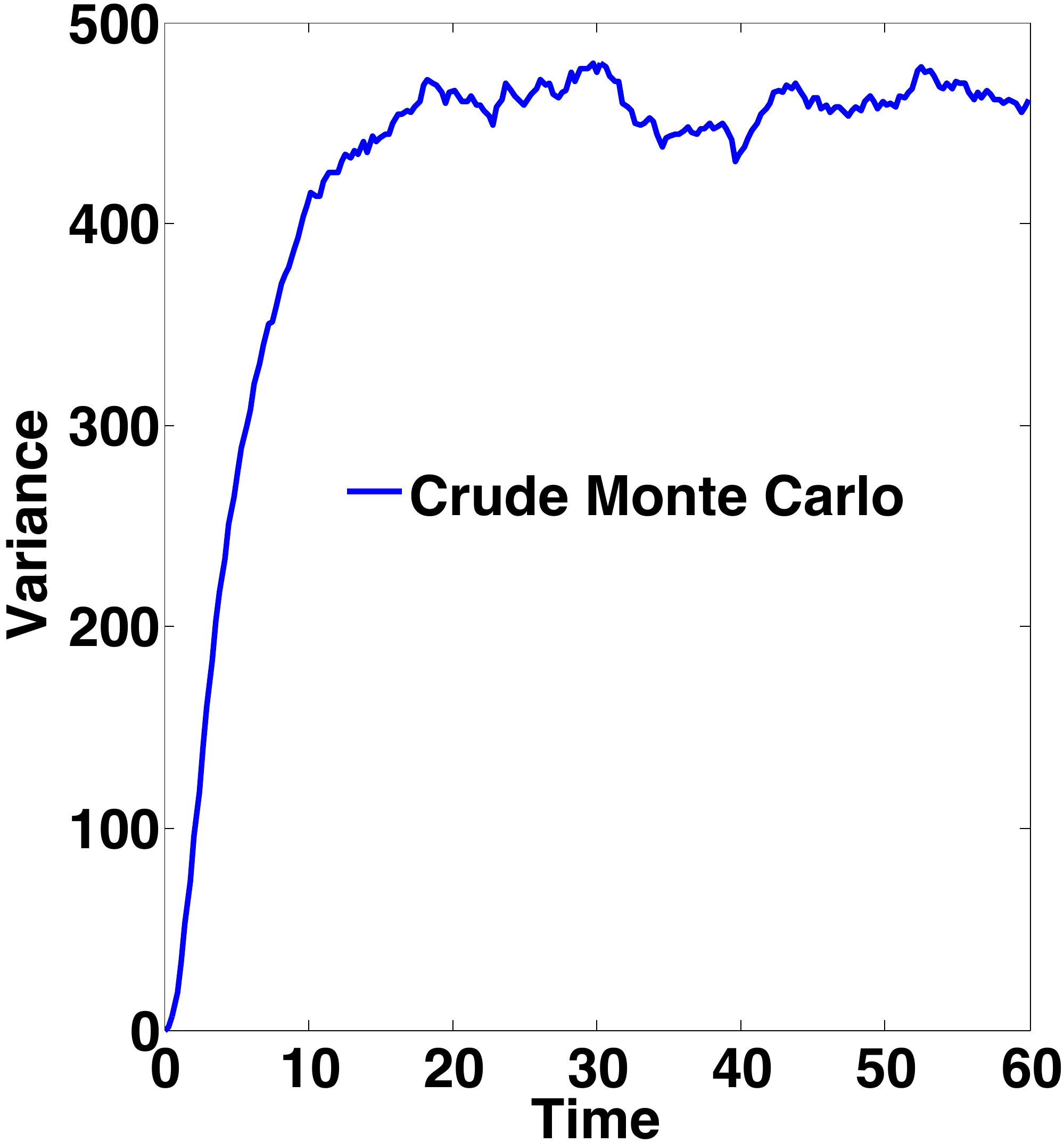}} 
        \qquad
      \subfloat[Girsanov Transformation]
      {\label{fig:Girsanov_gene}\includegraphics[height=2.7in]{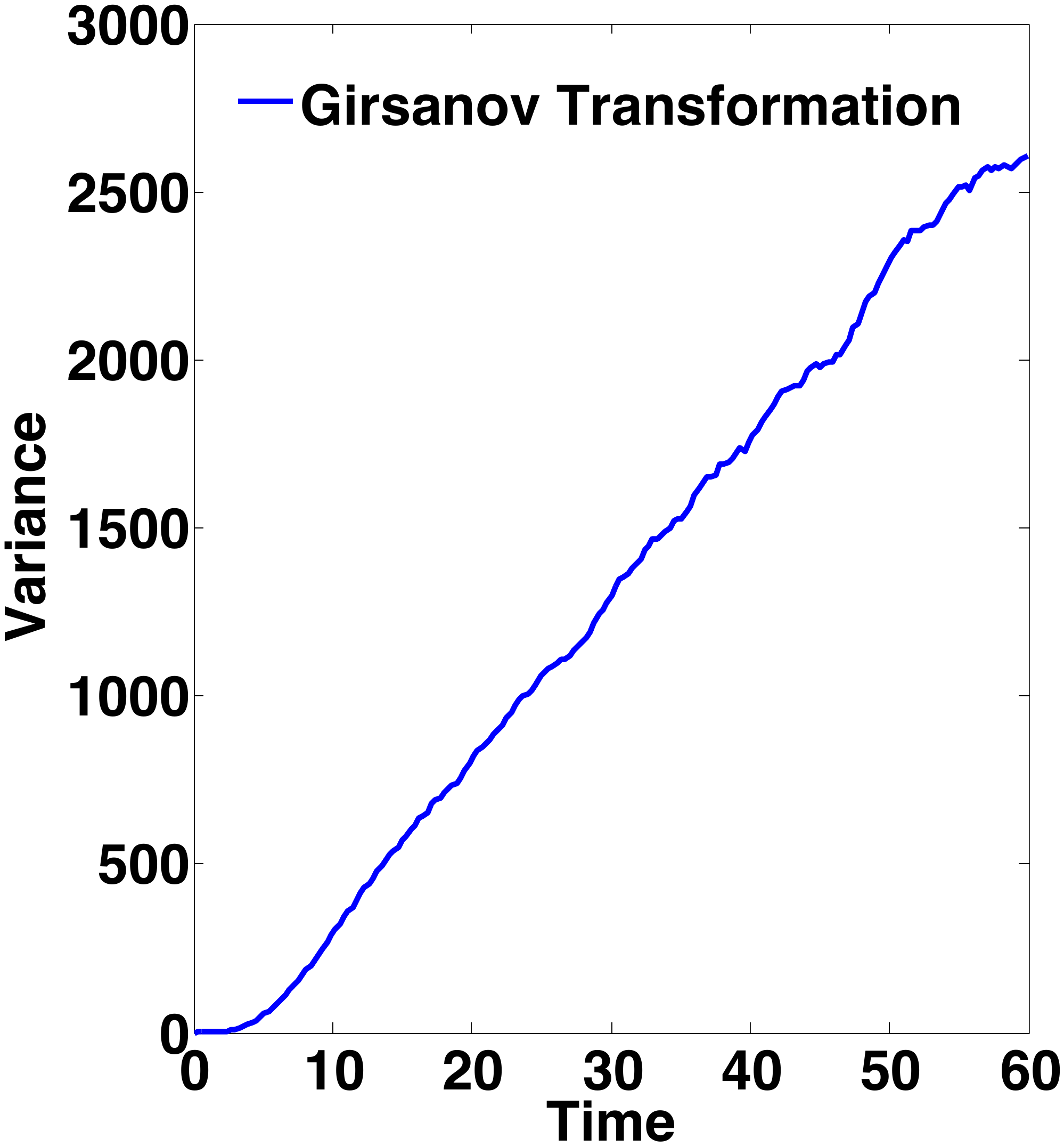}}
   \end{multicols}\end{center} %
    \caption{ Variance of the different estimators applied to \eqref{ex:Gene}.  For each, $R = $ 5,000 sample paths were used to construct the relevant estimators.  For each of the finite difference methods (figures \eqref{fig:CFD_CRP_gene} and \eqref{fig:cmc_gene}), a perturbation of $\epsilon = 1/40$ was used.  Note that the scales on the variance axis are dramatically different for the different methods.}
      \label{fig:VarTime1}
 \end{figure}
 
 \hfill $\square$
\end{num_ex}

\begin{num_ex}
We revisit  Example \ref{ex:mRNA}, which modeled an mRNA molecule being created and degraded (or an $M/M/\infty$ queue),
\begin{equation}\label{eq:MMinfinity}
	\emptyset \overset{\theta}{\underset{0.1}{\rightleftarrows}} M.
\end{equation}
We suppose that  we want to understand the sensitivity of the expected number of particles (or customers, in the queuing setting)  with respect to the parameter $\theta \approx 2$.  In Figure \ref{fig:2} we provide a plot of the variances of the different estimators as functions of time.  So as to demonstrate the different behaviors of the different estimators, the scales on both the time and variance axes are dramatically different for the different methods.  For each of the methods, we chose $R = $ 1,000, and used $\epsilon = 1/100$ for the perturbation methods.  Recall that in Example \ref{ex:mRNA}, we proved that the variance of the difference between $X^{\theta + \epsilon}$ and $X^\theta$ will converge to $\epsilon/0.1$ if they are coupled using \eqref{eq:main_coupling}.  Therefore, the variance of the estimator \eqref{eq:difference} will converge to 
\begin{equation*}
	\frac{\epsilon}{0.1}\frac{1}{\epsilon^2}\frac{1}{R} = \frac{1}{R} \frac{1}{0.1 \epsilon},
\end{equation*}
as $t \to \infty$.
In our case, $\epsilon = 1/100$ and $R = $ 1,000, and the above value is equal to one.  This predicted behavior is born out in Figure \eqref{fig:CFD_BD}.  Also in Example \ref{ex:mRNA}, we predicted (though did not prove) that after a long enough time the variance of both the common reaction path estimator (CRP) and Gillespie's algorithm plus common random numbers (CRN) should converge to the variance of the crude Monte Carlo estimator constructed with independent paths.  In essence, we are predicting that the processes will decouple after a long enough time and behave independently.  This behavior is demonstrated in Figures \eqref{fig:cmc_BD} and \eqref{fig:CRP_BD} (for CRP) and \eqref{fig:GILL_CRN} (for Gillespie + CRN), though we note that the time for  a full decoupling is quite large in this example.  Also note that we plotted the variance of the estimator for the common reaction path method both up to time $T = 100$ and $T = $ 10,000 so as to demonstrate the different behaviors exhibited.  Finally, we point out that the full ``decoupling'' of the CRP method described here does not seem to take place in Example 1.  The estimator built using the Girsanov transformation method exhibits a variance that grows linearly in time.

  \begin{figure}
    \begin{center}
      \subfloat[Crude Monte Carlo] {\label{fig:cmc_BD}
        \includegraphics[height=2.3in]{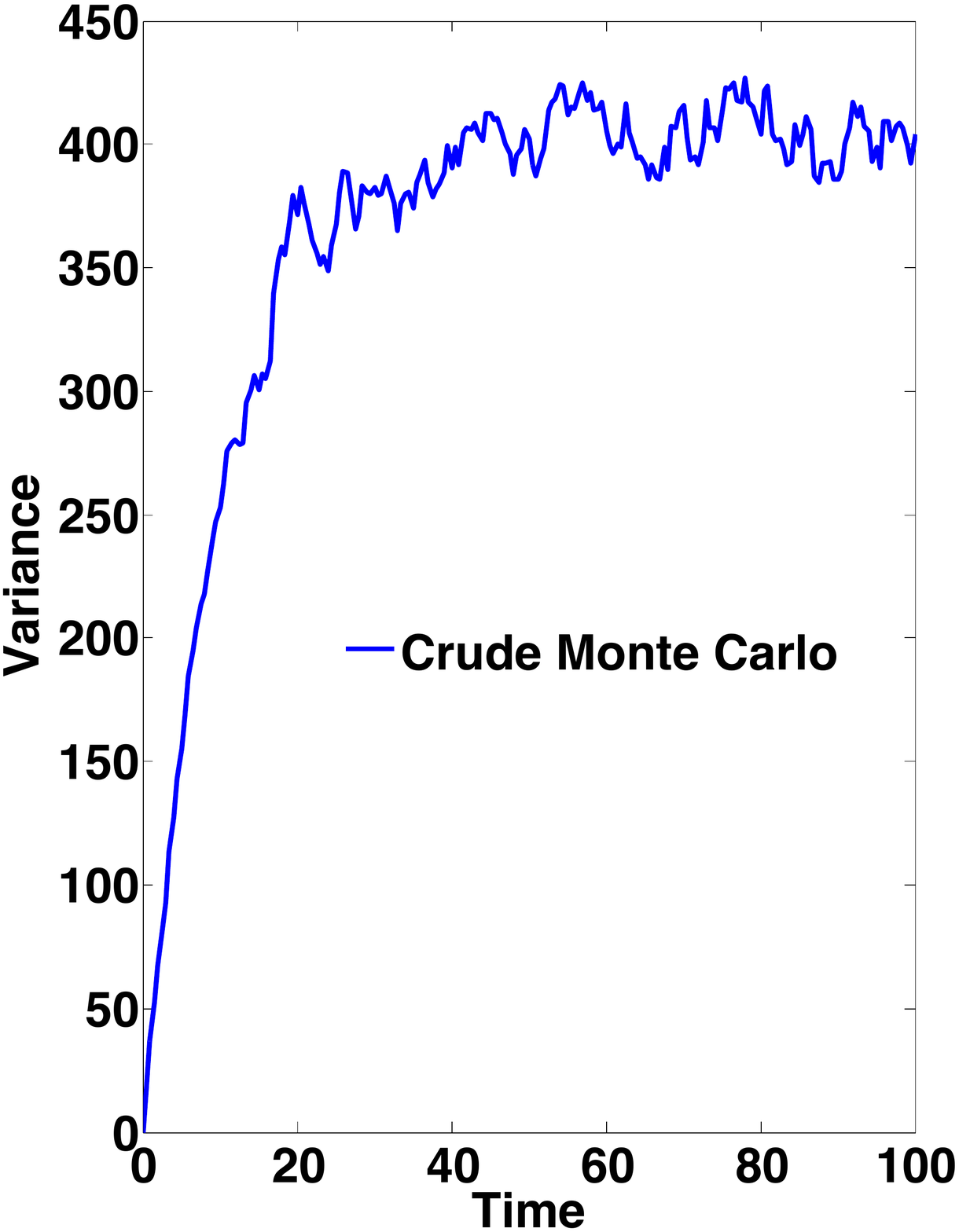}} 
        \qquad
      \subfloat[Girsanov Transformation]
      {\label{fig:Girsanov_BD}\includegraphics[height=2.3in]{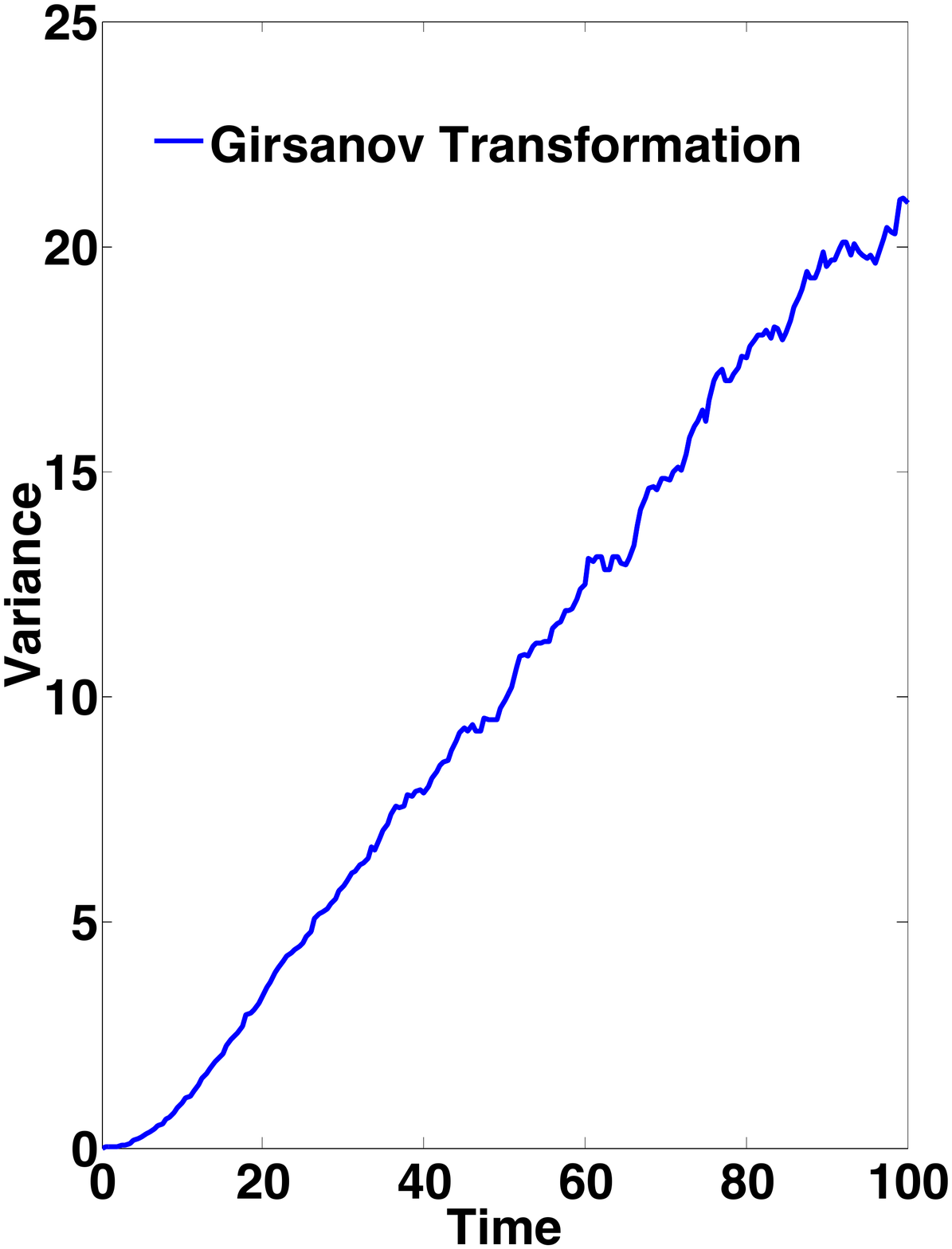}}
      \qquad
      \subfloat[Common Reaction Path, $T=10,000$]
      {\label{fig:CRP_BD}\includegraphics[height=2.3in]{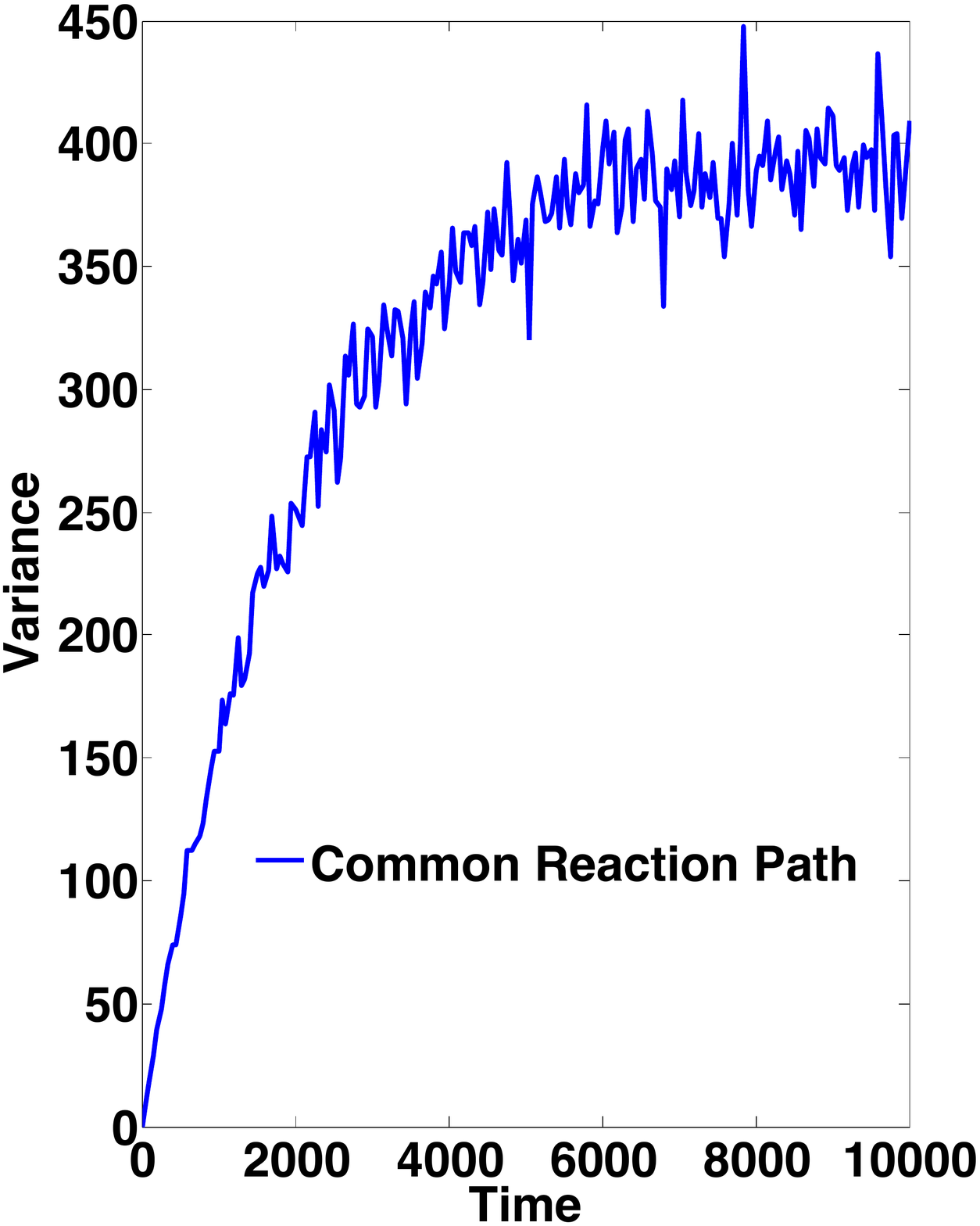}}\qquad
      \subfloat[Common Reaction Path, $T=100$]
      {\label{fig:CRP_100_BD}\includegraphics[height=2.3in]{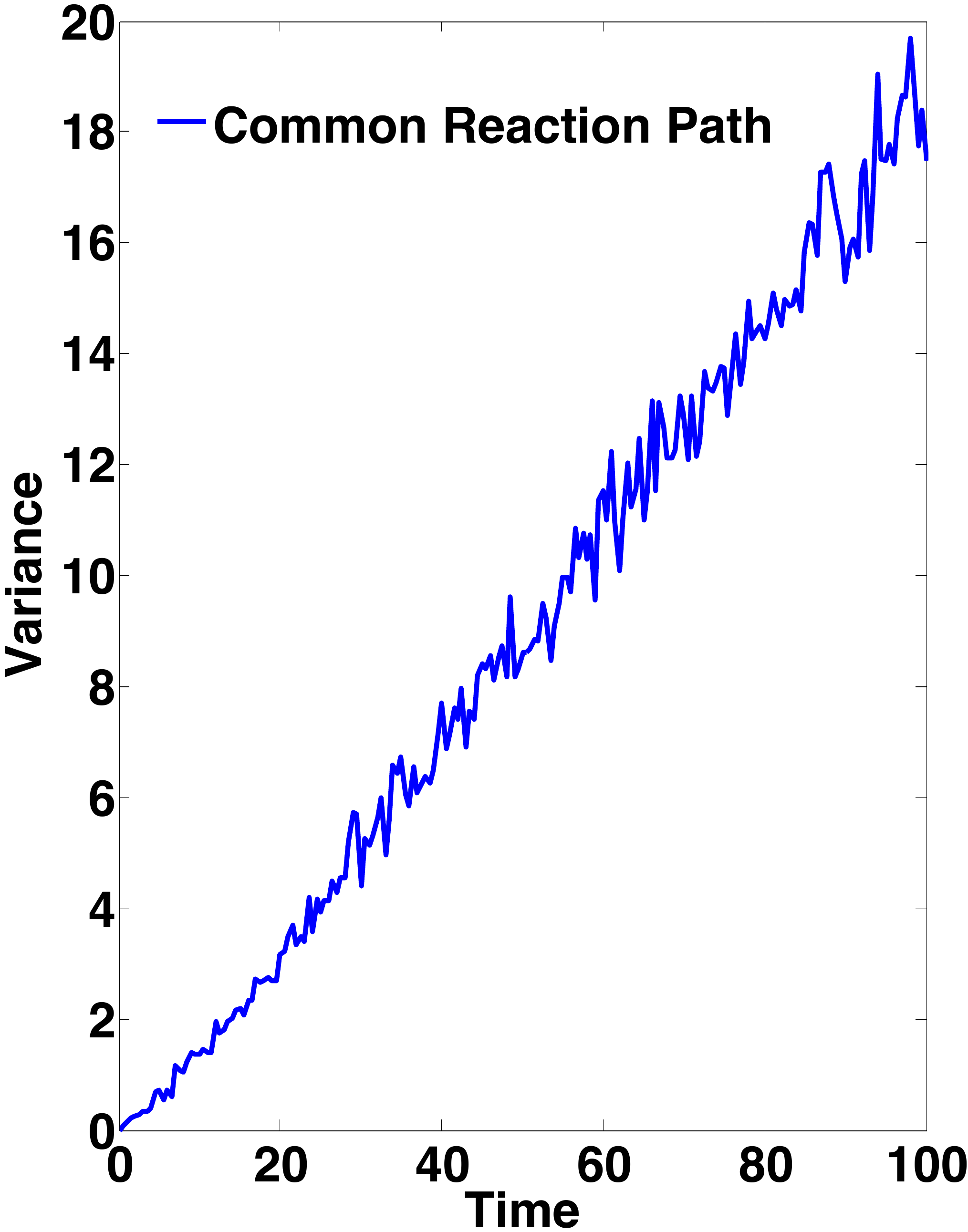}}\qquad
      \subfloat[Coupled Finite Differences]
      {\label{fig:CFD_BD}\includegraphics[height=2.3in]{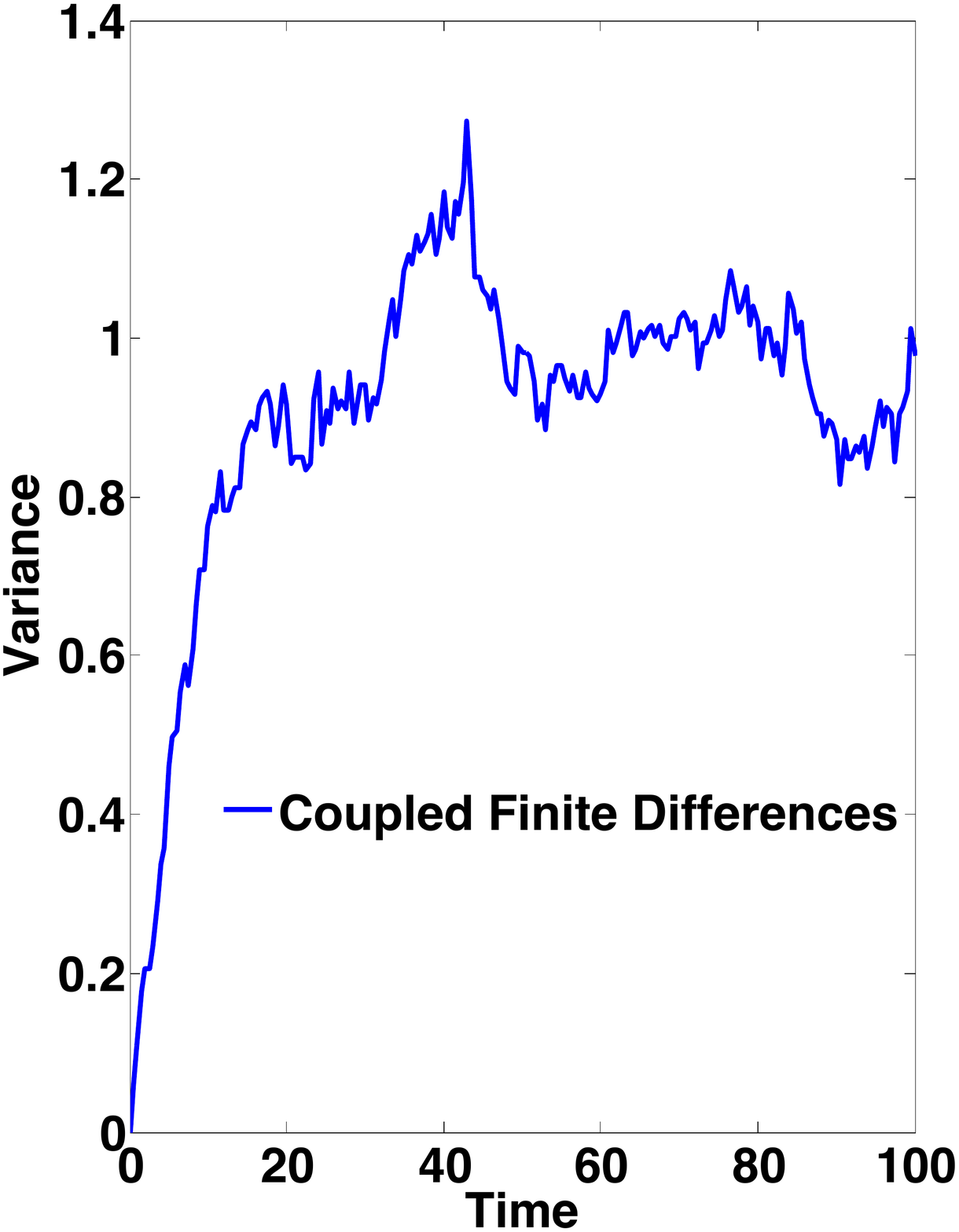}}\qquad
      \subfloat[Gillespie + Common Random Numbers]
      {\label{fig:GILL_CRN}\includegraphics[height=2.3in]{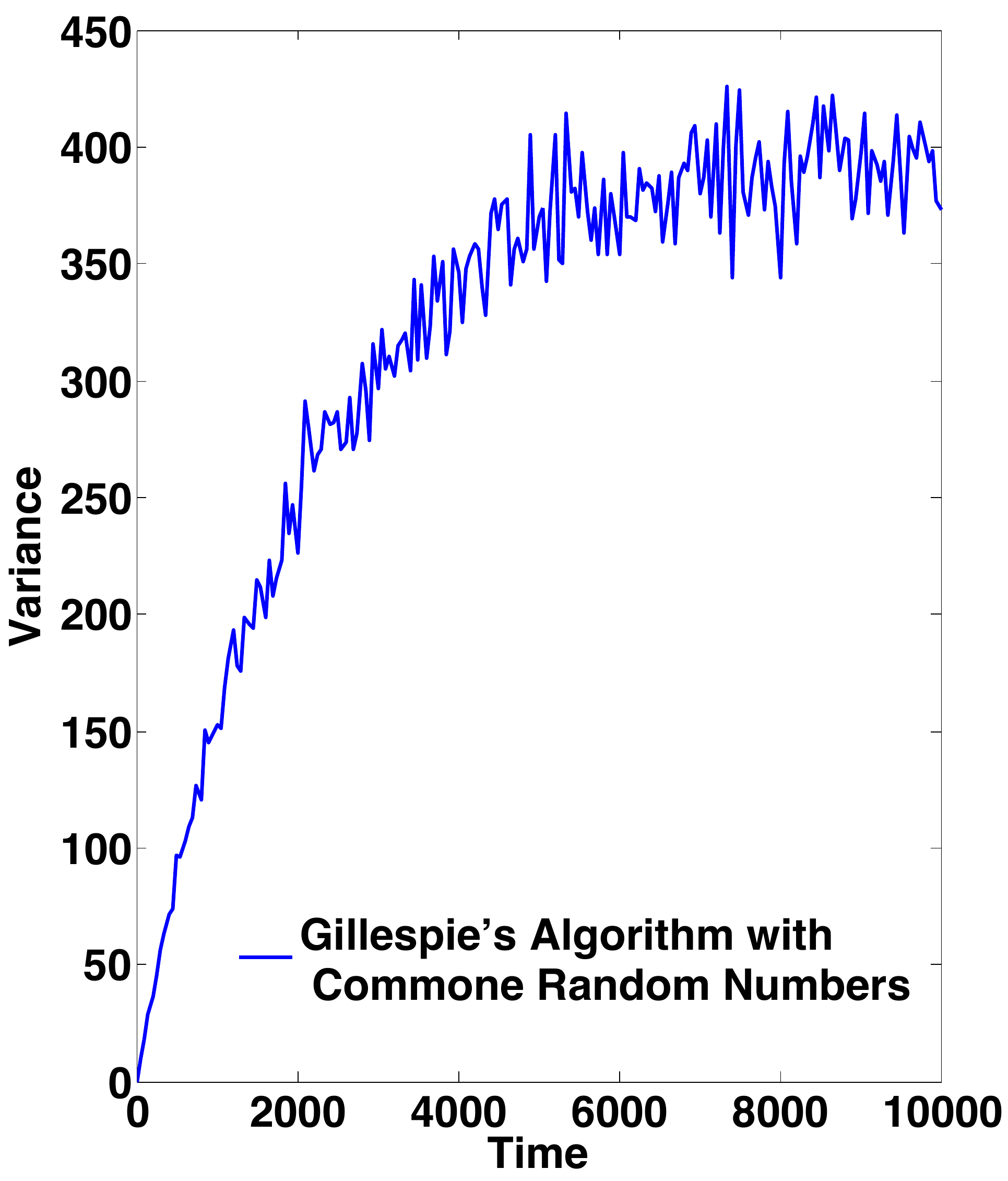}}
    \end{center} %
    \caption{ Variance of the different estimators applied to \eqref{eq:MMinfinity}.  $R = $ 1,000 sample paths were used to construct the relevant estimators.  For each of the finite difference methods (plots \eqref{fig:cmc_BD}, \eqref{fig:CRP_BD}, \eqref{fig:CRP_100_BD}, \eqref{fig:CFD_BD}, and \eqref{fig:GILL_CRN}), a perturbation of $\epsilon = 1/100$ was used.  So as to demonstrate the different behaviors of the different estimators, the scales on both the time and variance axes are dramatically different for the different methods.  Also, note that we plotted the variance of the estimator for the common reaction path method both up to time $T = 100$ and $T = $ 10,000 so as to demonstrate the different behaviors exhibited.  Note that the CRP and Gillespie + Common Random Number methods appear equivalent for this example.}
      \label{fig:2}
 \end{figure}

 Next, we considered the sensitivity to the decay parameter at $0.1$.  That is we considered
 \begin{equation}\label{eq:MMinfinity2}
	\emptyset \overset{2}{\underset{\theta}{\rightleftarrows}} M,
\end{equation}
with $\theta = 0.1$.
 In Figure \ref{fig:CFD_CRP_BD_New} we provide a plot of the variance of the Coupled Finite Difference estimator versus the Common Reaction Path estimator.  Each plot was generated using 1,000 sample paths in which a perturbation of $\epsilon = 1/100$ was used.  We again see the lower variance exhibited by the Coupled Finite Difference Estimator, though the difference is now less dramatic.
 \hfill $\square$
   \begin{figure}
    \begin{center}
      \includegraphics[height=3.0in]{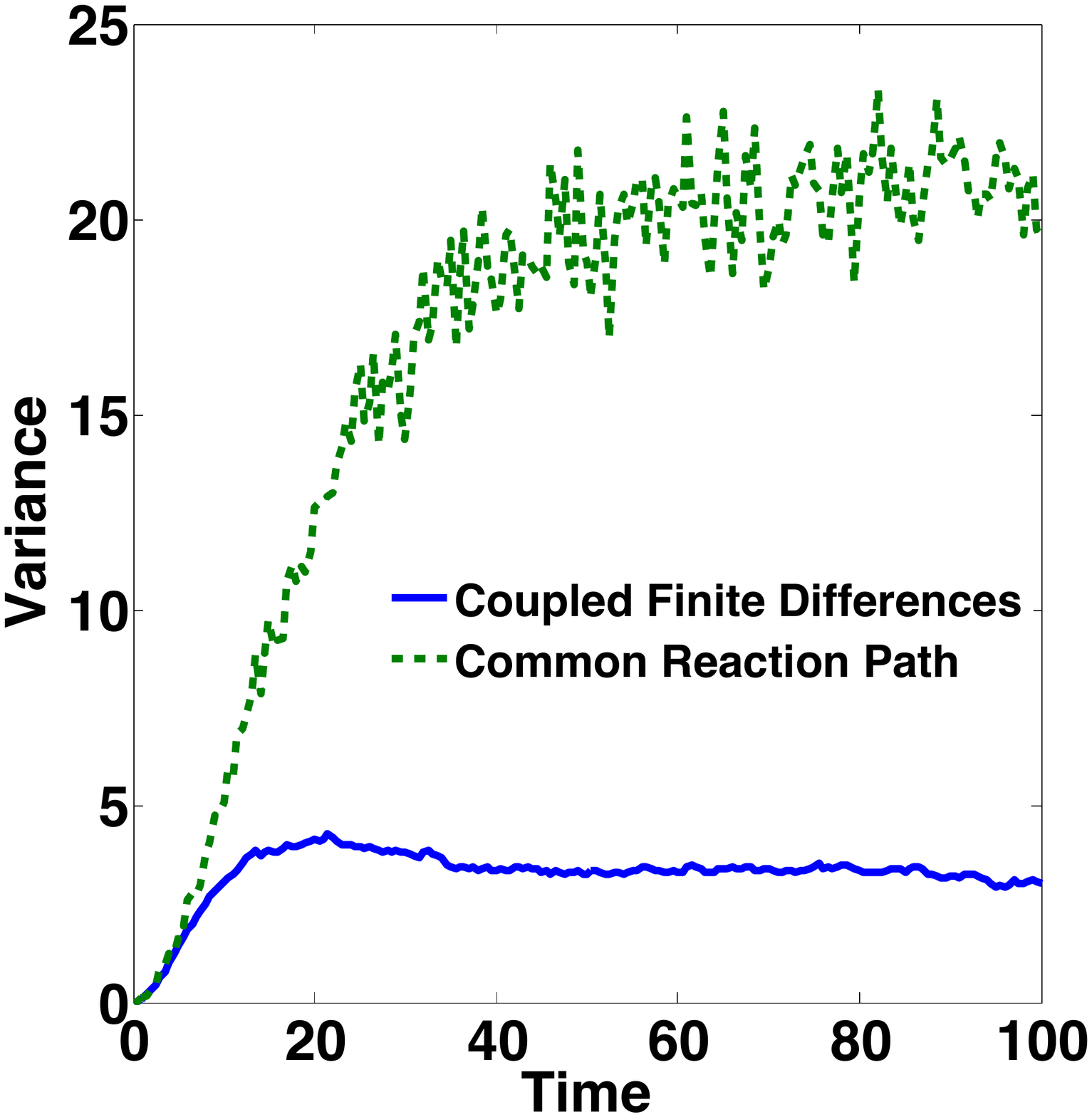}
    \end{center} %
    \caption{ Time plot of the variance of the Coupled Finite Difference estimator versus the Common Reaction Path estimator for the  model \eqref{eq:MMinfinity2} with decay rate perturbed.  Each plot was generated using 1,000 sample paths. A perturbation of $\epsilon = 1/100$ was used.}
      \label{fig:CFD_CRP_BD_New}
 \end{figure}
 \end{num_ex}

 \begin{num_ex}
 	We consider a model for a genetic toggle switch found in \cite{Khammash2010}:
	\begin{align}\label{ex:toggle}
		\emptyset \underset{\lambda_2(X)}{\overset{\lambda_1(X)}{\rightleftarrows}} X_1,
	 \quad \emptyset \underset{\lambda_4(X)}{\overset{\lambda_3(X)}{\rightleftarrows}} X_2,
\end{align}
 with intensity functions
 \begin{align*}
 	\lambda_1(X(t)) &= \frac{\alpha_1}{1 + X_2(t)^{\beta}}, \quad \lambda_2(X(t)) = X_1(t)\\
	\lambda_3(X(t)) &= \frac{\alpha_2}{1 + X_1(t)^{\gamma}}. \quad \lambda_4(X(t)) = X_2(t),
 \end{align*}
 and parameter choice 
 \[
 	\alpha_1 = 50, \quad \alpha_2 = 16, \quad \beta = 2.5, \quad \gamma = 1.
 \]
 We begin the process with initial condition $[0, 0]$ and consider the sensitivity of $X_1$ as a function of $\alpha_1$.  In Figure \ref{fig:toggle} we provide a plot of the variance of the Coupled Finite Difference estimator versus the Common Reaction Path estimator as a function of time.  Each plot was generated using 10,000 sample paths in which a perturbation of $\epsilon = 1/10$ was used.  We see that the Coupled Finite Difference method performs substantially better for times $t> 5$, whereas the Common Reaction Path method performs better for shorter times, $t<5$.
 
  \begin{figure}
    \begin{center}
    \subfloat[Variance to time $T=40$]
    {\label{fig:toggle1}\includegraphics[height=3.0in]{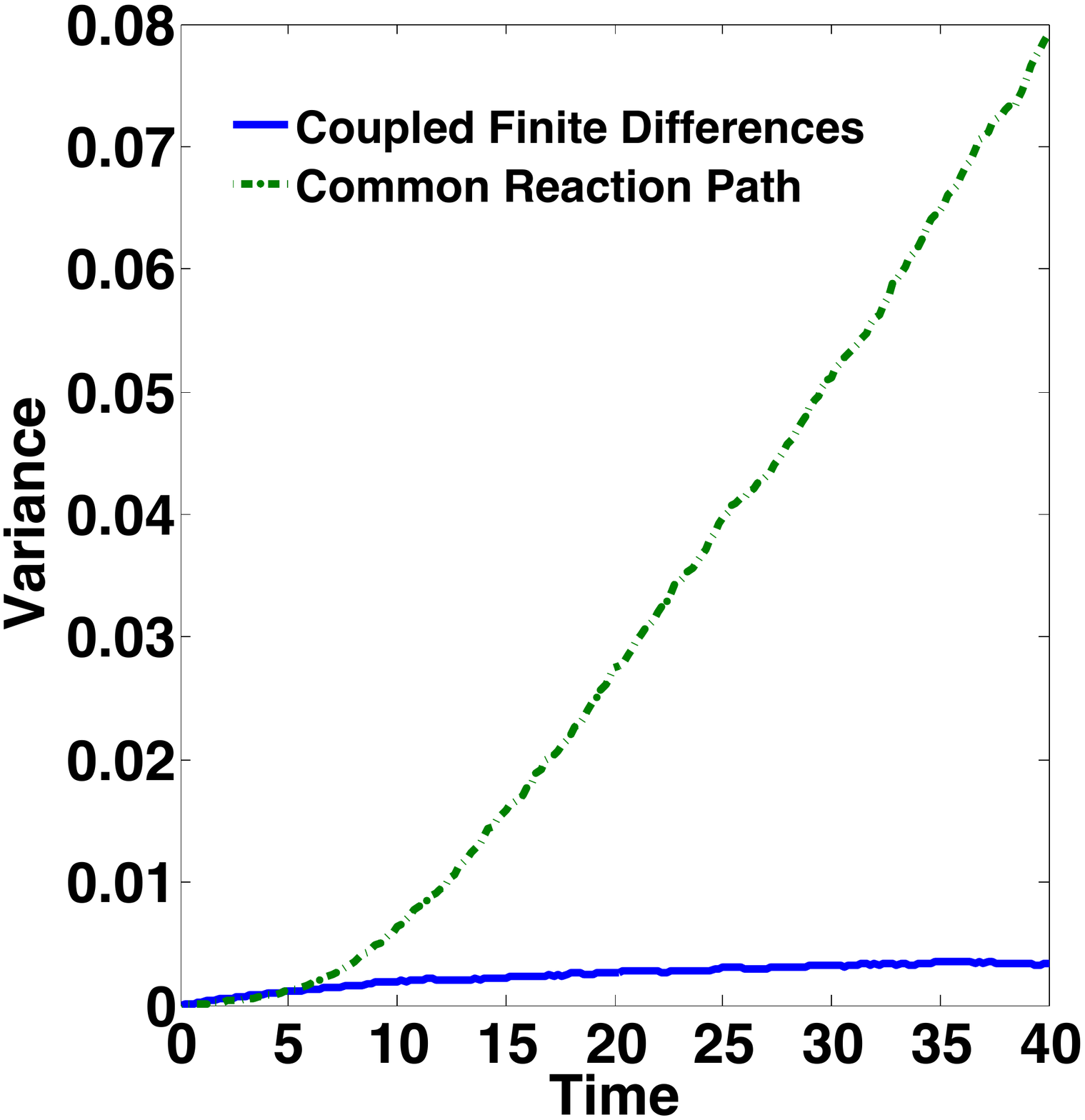}} \quad
     \subfloat[Variance to time $T=7$]
    {\label{fig:toggle2}\includegraphics[height=3.0in]{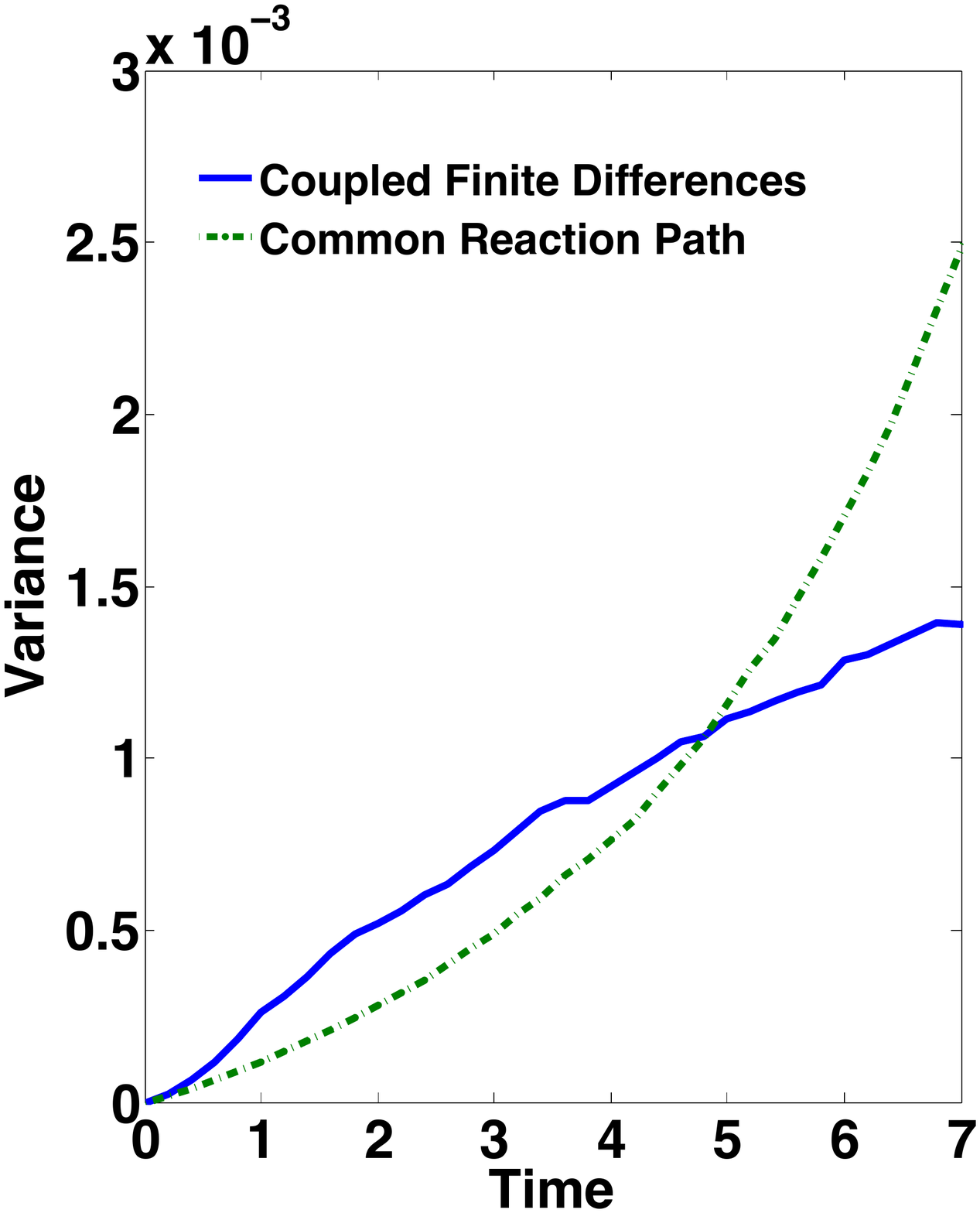}}
    \end{center} %
    \caption{ Time plot of the variance of the Coupled Finite Difference estimator versus the Common Reaction Path estimator for the  model \eqref{ex:toggle}.  Each plot was generated using 10,000 sample paths. A perturbation of $\epsilon = 1/10$ was used. It is worth noting that here the Common Reaction Path estimator outperforms the Coupled Finite Difference estimator for times $t < 5$, though CFD still greatly outperforms CRP for larger times.}
      \label{fig:toggle}
 \end{figure}
 
 \end{num_ex}

%
%

\bibliographystyle{amsplain} \bibliography{sensitivity}

\end{document}